\documentclass[reqno]{amsart}
\usepackage{amssymb,amsmath,hyperref}
\usepackage{amsrefs}
\usepackage[foot]{amsaddr}
\usepackage{bbold,stackrel}

\newtheorem{theorem}{Theorem}[section]

\newtheorem{lemma}[theorem]{Lemma}
\newtheorem{definition}[theorem]{Definition}
\newtheorem{corollary}[theorem]{Corollary}

\newtheorem{example}{Example}

\newtheorem{observation}{Observation}
\newtheorem{remark}[theorem]{Remark}

\newtheorem{problem}{Problem}

\newtheorem{comment}{Comment}

\newcommand{\R}{{\mathbb R}}
\newcommand{\N}{{\mathbb N}}

\newcommand{\su}{
{\bf S}}

\newcommand{\ci}{\mathcal I}
\newcommand{\what}{\;\hat{}\;}

\newcommand{\WO}{\textup{\textsf{WO}}}
\newcommand{\HBU}{\textup{\textsf{HBU}}}
\newcommand{\Ca}{\textup{\textsf{C}}}
\newcommand{\Ba}{\textup{\textsf{B}}}
\newcommand{\Po}{{\mathcal P}}
\newcommand{\LL}{\textup{\textsf{L}}}

\begin{document}

\title{Computability and  Non-monotone Induction}

\author{Dag Normann}
\address{Department of Mathematics, The University 
of Oslo, P.O. Box 1053, Blindern N-0316 Oslo, Norway}
\email{dnormann@math.uio.no}
\thispagestyle{empty}

\newpage
\thispagestyle{empty}

\begin{abstract}Non-monotone inductive definitions   were studied in the late 1960's and early 1970's with the aim of understanding connections between the complexity of the formulas defining the induction steps and the size of the ordinals measuring the duration of the inductions. In general, any type 2 functional will generate an inductive process, and in this paper we will   view non-monotone induction as a functional of type 3. We investigate the associated computation theory inherited from the Kleene schemes and we investigate the nature of the associated companion of sets with codes computable in non-monotone induction. The interest in this functional  is motivated from observing that constructions via non-monotone induction appear as natural in classical analysis in its original form.

There are two groups of results: We establish strong closure properties of the least ordinal without a code computable in non-monotone induction, and we provide a characterisation of the class of functionals of type 3 computable from non-monotone induction, a characterisation  in terms of sequential operators working in transfinite time.
We will also see that the full power of non-monotone induction is required when this principle is used to construct functionals witnessing the compactness of the Cantor space and of closed, bounded intervals.
\end{abstract}

\maketitle

\section{Introduction}\label{1.}
\subsection{Motivation and history}With the introduction of set theory in the second half of the 19th century, mathematicians had more tools in their toolbox than before, they had a richer language in which to express mathematical properties, but they also had tools like transfinite recursion and the use of the axiom of choice. One of these tools, inspired from the new ordinal numbers introduced by Cantor,  is non-monotone induction over the set of integers, seen as an operator of order four, or of type 3 in the terminology of type theory. 

It is worth noticing that the set-theoretical language mostly used at the time is of third order, while coding is needed to capture the same concepts in second order arithmetic (SOA). In a series of papers\cite{P1,P2,P3,P4,P5,P6,P7}, Sam Sanders and the author have investigated the logical and computability strength of some of the results using such  tools, when expressed in a language close to how it was originally done. 
\smallskip

Non-monotone inductive definitions  were studied in the late 1960's and early 1970's, but the general interest has been low since then. Examples of papers on the subject are \cites{AR,DS,RA,staal}. The inductive definitions were classified according to the complexity of the formulas defining them, and the key property of interest was the complexity of the corresponding closure ordinals. This could be expressed in terms of reflection properties as in \cite{RA} or by comparing classes of closure ordinals as in \cite{staal}.
\smallskip

In this paper we
will view non-monotone inductive definability over $\N$ via a functional $\ci$ of type 3 (Definition \ref{def.ind}), and investigate the strength of Kleene computability (Definition \ref{Kleene}) relative to $\ci$ . As there is no justifiable Church-Turing thesis for the computability theory of higher order functionals, Kleene computability is just one possible model, but since this model has proved to be fruitful for the analysis of discontinuous functionals of type 2, and for computability relative to the Superjump as defined by Gandy \cite{Gandy}, see Harrington \cite{H72}, Kleene computability is a natural model for the investigation of the computational strength of non-monotone induction.
\smallskip

The motivation for bringing up non-monotone induction once again is the observation that this functional represents a natural upper complexity-bound for other functionals appearing as realisers for classical theorems such as the Heine-Borel theorem and the Baire Category theorem, when these theorems are formalised in a set-theoretic language and not within the restricted language of second order arithmetic.
\smallskip

The first application of non-monotone inductive definitions known to the author is due to E. Borel \cite{B95}. The motivation of Borel was to give a \emph{direct} proof of the theorem now known as the Heine-Borel theorem.  The assumption was that we are given a way to associate an open neighbourhood $O_x$  to each $x$ in a closed interval $[a,b]$ and the claim was that we can then explicitly find a finite sub-covering. In the terminology of today, Borel constructed a functional taking the map $x \mapsto O_x$ as the argument and yielding a finite subcovering as the value. The definition of this functional is by transfinite recursion, building up finite subcoverings of larger and larger closed subintervals, a construction that can be viewed as a simultaneous non-monotone inductive definition of  Dedekind cuts for numbers $c \leq b$ and finite subcoverings of each closed interval $[a,d]$ for $d < c$. In \cite{P1}, a realiser  $\Theta_0$ of the uncountable Heine-Borel theorem ($\HBU$) is defined. This realiser selects a finite set $x_1 , \ldots , x_n$ such that the corresponding open neighbourhoods form a subcovering. It is proved in \cite{P1} that $\Theta_0$, in conjunction with $^2E$, computes the Suslin functional (see below), and in Normann \cite[Theorem 1(c)]{N18} it is shown that any realiser $\Theta$  of $\HBU$ as above , in conjunction with the Suslin operator, computes the functional $\ci$ to be defined below. We will slightly improve this theorem, see Section \ref{7.2}.
\smallskip

Realisers $\Xi$ for the Lindel\"of lemma for Baire space $\N^\N$ (homeomorphic to the irrationals) is one  class of functionals discussed in \cite{N18}, where  it is proved in Theorem 1 that any such realiser will compute $\ci$ and that there is at least one such realiser computable in $\ci$. Thus non-monotone induction reflects the complexity of witnessing the Lindel\"of lemma in this special case.

\smallskip

In \cite{P6} the aim is to investigate real line topology with the purpose of classifying the complexity of theorems and concepts in terms of their \emph{reverse mathematics} and \emph{computational complexity}. Representations of open sets, such as being countable unions of rational neighbourhoods, are based on mathematical insight, and analysing the logical and computational strength of such insight is part of the aim of \cite{P6}. Given representations of open sets as in classical reverse mathematics, using second order arithmetic, the Baire Category Theorem is effective in the sense that given ( a representation of) a sequence of dense open sets we can compute a fast-converging Cauchy-sequence for a point in the intersection. In \cite[Theorem 6.5]{P6} it is proved that, using non-monotone induction, we can find a functional $\xi$ taking a sequence $\{X_k\}_{k \in \N}$ of subsets of $\R$ as arguments and yielding an $x \in \R$ as value, such that whenever each $X_k$ is dense and open then \[\xi(\{X_k\}_{k \in \N}) \in \bigcap_{k \in \N}X_k.\]
In \cite[Theorem 6.6]{P6} it is proved that no such functional $\xi$ can be computable in any functional of type 2, but it remains open to decide if the full power of non-monotone induction is needed for obtaining a functional $\xi$ like this.
\subsection{Overview and results}
In Section \ref{2.} we will define the functional $\ci$ that is our main subject of investigation, and we will define the Kleene-computations via the schemes S1-S9 with $\ci$ as the one argument of type 3. We observe two interpretations of these schemes, one where we follow Kleene and restrict the application scheme S8 to total inputs and one where we relax on this requirement. We show that the two interpretations lead to the same class of  functions of type 1 computable in $\ci$. We use this to prove what is known as \emph{stage comparison} and \emph{Gandy selection} for the interpretation using partial inputs.
\smallskip

In Section \ref{Comp} we investigate the least ordinal $\pi$ not computable in $\ci$, and the associated companion $\LL_\pi$. We prove that the set of codes $f \in \WO$ of $\pi$ is not computable in $\ci$, and thus in particular not a $\Pi^1_1$-set (Corollary \ref{cor.car.1}). We also establish a number of reflection properties for $\pi$.
\smallskip

Section \ref{3.} is a preparation for Section \ref{5.}. In Section \ref{3.} we introduce what we call \emph{hyper-sequential procedures} and in Section \ref{5.} we narrow down this concept to \emph{inductive procedures}. These procedures model nested systems of non-monotone inductions, using our new concept of \emph{ blockings} to organise the nesting. The inductive procedures can be used to characterise the class of functionals of type 3 computable in $\ci$.
\smallskip

In Section \ref{chapt.7} we look at some of the functionals serving as realisers for classical theorems in analysis, primarily theorems where the proof in some way depends on the compactness of the unit line or Cantor space. We will see that when such realisers  are constructed in a natural way, they implicitly have the full power of non-monotone induction. In conjunction with the Suslin functional $\su$, all realisers of the theorems in question will compute $\ci$.
We will illustrate how to use compactness for computing $\ci$  in the proof of Lemma \ref{lemma.pin}, a lemma that  is a slight improvement of \cite[Theorem 1 (c)]{N18}.
\smallskip

In Section \ref{sec8} we briefly discuss what it means to relativise these results to functionals of type 2 and in Section \ref{sec9} we summarise the paper and discuss a few open problems.

\section{Non-monotone induction and computability}\label{2.}
\subsection{Inductive definitions}\label{2.1}
Mathematically we can identify the Cantor set  $\Ca  = \{0,1\}^\N$ with the powerset $\Po(\N)$ of the integers, where we identify a set with its characteristic function. In this paper, we will use both notations, as it sometimes is essential that we consider the set as $\Ca$, the compact set, and sometimes consider the set $\Po(\N)$ where the inclusion ordering is essential.  This view will be relevant when we define non-monotone induction, but mathematically we use $\Ca$ as the formal definition of the set under consideration, and treat it as $\Po(\N)$ when this is convenient. When elements of $\Ca$ are viewed as characteristic functions, the point-wise ordering $\leq$ coincides with the inclusion ordering $\subseteq$.
\begin{definition}\label{def.ind}{\em Let $F:\Ca \rightarrow \Ca$ be a functional of type 2. 
\begin{itemize}
\item[a)]We  view $F$ as an inductive definition, defining the increasing sequence $f_\beta$ in $\Ca$  where $\beta$ runs over the countable ordinals, by transfinite recursion as follows:

\begin{enumerate}
\item $f_0$ is the constant zero
\item $f_{\beta + 1} = \max\{f_\beta , F(f_\beta)\}$
\item If $\beta$ is a limit ordinal, $f_\beta = \sup_{\gamma < \beta} f_\gamma$.

\end{enumerate}

\item[b)]
There will, for cardinality reasons, be a least countable ordinal $\alpha_F$ such that $f_{\alpha_F} = f_{\alpha_F + 1}$. Then $\alpha_F$ is the least ordinal $\alpha$ such that $F(f_\alpha) \leq f_\alpha$.
We let $\ci$ be defined by
$\ci(F) = f_{\alpha_F}$,
with the notation introduced above.
\end{itemize}
If we need to point to the functional $F$, we write  $f^F_\beta$.}\end{definition}

\begin{remark}\label{Remark.2.1}{\em We are not fully in the realm of Kleene-computability, since this is developed for total functionals of pure type only. However, if $G$ is of pure type 2, we may consider $G$ as a code for 
\[F_G(f)(n) = \min\{G(n\what f), 1\},\]
where $f \in \Ca$ and with the standard concatenation-understanding of $n\what f \in \N^\N$. Using standard coding, we my also consider $\ci$ as a functional of type 3. For the sake of readability, we prefer to use a customised version of Kleene's definition, as defined in Section \ref{2.kleene}, when we investigate the computational strength of $\ci$.

}\end{remark}
\begin{example}{\em We view $\Ca$ as the powerset of $\N$ and let $G:\Ca \rightarrow \N$. For pure cardinality reasons, there must be $A \neq B \subseteq \N$ such that $G(A) = G(B)$, and, by the axiom of choice, there will be a functional $\Phi$ such that for every $G$, $\Phi(G)$ is such a pair. Now, the axiom of choice is not needed for this, as will be seen from an easy application of $\ci$:
\medskip

Given $G:\Ca \rightarrow \N$, let $F_G$ be defined by $F_G(A) = A \cup \{G(A)\}$. We then see that the transfinite iteration of $F_G$ will generate a strictly increasing sequence of sets $\{A_\beta\}_{\beta \leq \alpha}$ exactly until we have an $\alpha$, and a $\beta < \alpha$, such that $G(A_\beta) = G(A_\alpha)$.
\medskip

In \cite{P7} the complexity of such functionals $\Phi$ witnessing that there is no injection from $\Ca$ to $\N$ is studied in more detail, and it is proved that no such functional can be computed from an object of type two.

}\end{example}
\subsection{Kleene computability}\label{2.kleene}

Kleene \cite{kleene} defined a relation $\{e\}(\vec \Phi) = a$, in the form of a positive  inductive definition with nine cases, where $e$ is an \emph{index}, a natural number that serves as a G\"odel number for a generalised algorithm, and $\vec \Phi$ is a sequence of functionals of pure types in the type-structure of total functionals. The nine cases in the definition are called \emph{schemes} and are numbered as S1 - S9. For a recent introduction to Kleene computability, see Longley and Normann \cite[Chapter 5]{LN}.
\medskip

In this section we will mainly be concerned with computations of the form 
\[\{e\}(\ci  , \vec F, \vec f , \vec a)\] where  $\vec F $ is a sequence of functionals of type 2, $\vec f$ is a sequence of functions of type 1 and $\vec a $ is a sequence from $\N$. In Definition \ref{Kleene} we will restrict S1 - S9 to this case. In Section \ref{3.2} we will give a more general version of S8, accommodated to the content of that section. Our  version of S8 here, when restricted to the use of $\ci$ as the only object of type 3, will be equivalent to using the version of S8 in Section \ref{3.2} to the functional of pure type 3 that will represent $\ci$.

\begin{definition}\label{Kleene}{\em Using transfinite recursion, we define the relation $\{e\}(\ci, \vec F , \vec f , \vec a) = c,$
where $\ci$ is as defined, $\vec F = (F_1 , \ldots , F_m)$ is a sequence from $\N^\N\rightarrow \N$, $\vec f = (f_1 , \ldots , f_n)$ is a sequence from $\N^\N$, $\vec a = (a_1 , \ldots , a_k)$ is a sequence from $\N$ and $c \in \N$, as follows.
\begin{itemize}
\item[S1] If $e = \langle 1 \rangle$, then $\{e\}(\ci ,\vec F ,\vec f , \vec a) = a_1 + 1$.
\item[S2] If $e = \langle 2,q\rangle$, then $\{e\}(\ci,\vec F , \vec f , \vec a) = q$.
\item[S3] If $e = \langle 3\rangle$, then $\{e\}(\ci,\vec F , \vec f , \vec a) = a_1$.
\item[S4] If $e = \langle 4, e_1,e_2\rangle$, $\{e_2\}(\ci, \vec F, \vec f , \vec a) = b$ and $\{e_1\}(\ci ,\vec F, \vec f , b , \vec a) = c$, then $\{e\}(\ci  ,\vec F,  \vec f , \vec a) = c$.
\item[S6]If $e = \langle e_1 ,\tau_1 , \tau_2, \tau_3, \rangle$, where $\tau_1$ ,  $\tau_2$ and $\tau_3$ are permutations of (the index sets for) the input sequences $\vec F$, $\vec f$ and $\vec a$, then $\{e\}(\ci , \vec F, \vec f , \vec a) = \{e_1\}(\ci , \vec F_{\tau_1}, \vec f_{\tau_2} ,\vec a_{\tau_3})$.
\item[S7] If $e = \langle 7 \rangle$, then $\{e\}(\ci ,\vec F, \vec f , \vec a) = f_1(a_1)$.
\item[S8] For this scheme  there will be subcases, one for each type $> 1$. For us, there will be two subcases, where the case for type 3 is where we adjust the definition to application of $\ci$:
\begin{itemize}
\item[2.] If $e = \langle 8,2,d\rangle$ then $\{e\}(\ci ,\vec F, \vec f , b , \vec a) = F_1(g)$ when
$g(a) = \{d\}(\ci , \vec F , \vec f , a , \vec a)$ is a total function. We write
\[\{e\}(\ci , \vec F , \vec f , \vec a) = F_1(\lambda a. \{d\}(\ci , \vec F , \vec f , a, \vec a)).\]

\item[3.]If $e = \langle 8,3,d\rangle$ we let
$\{e\}(\ci ,\vec F,  \vec f , b , \vec a) = \ci(F_G)(b)$ where $G(f) = \{d\}(\ci ,\vec F,  f , \vec f, \vec a ).$
\end{itemize}

\item[S9] If $e = \langle 9 \rangle$ then $\{e\}(\ci ,\vec F,  \vec f , d , \vec a) = c$ if $\{d\}(\ci,\vec F , \vec f , \vec a) = c$.
\end{itemize}
}\end{definition}
\begin{remark}{\em We have excluded  S5, the scheme of primitive recursion, from our definition. There are two reasons for this. The main reason is that one may prove the recursion theorem on the basis of the other schemes, and thus S5 will be redundant. The other reason is that, since recursion is iterated composition, all arguments involving S5 that we need will be covered by how we deal with S4.}\end{remark}

Kleene computability inherits several of the key properties of classical computability, such as the $\textsf{S}_{n,m}$-theorem and the recursion theorem. The existence of universal algorithms is \emph{axiomatised} in the form of S9. In the sequel, we will assume familiarity with these basic properties.
\subsection{The computability theory of $\ci$}

We first prove that the prototype of discontinuity is computable in $\ci$.
\begin{definition}{\em
We define the functional $^2E$ of type 2 by
\[^2E(f) = \left\{ \begin{array}{ccc} 0 & {\rm if} & \forall k (f(k) = 0) \\ 1 & {\rm if} & \exists k (f(k) > 0)\end{array} \right. \] }\end{definition}

\begin{lemma} The functional $^2E$ is computable in $\ci$.
\end{lemma}
\begin{proof}
Given $f \in \N^\N$, we want to decide if $\exists k (f(k) > 0)$. Let 
\[F_f(A) = \{k : f(k) > 0\} \cup \{k : k+1 \in A\}.\]
Then $\exists k( f(k) > 0)$ if and only if $0 \in \ci(F_f)$. \end{proof}
\begin{remark}{\em $^2E$ is sometimes denoted  as $\exists^2$, and is equivalent, within S1 - S9, to Feferman's $\mu$.}\end{remark}
The Suslin functional $\su$ is defined by 
\[\su(f) = \left\{ \begin{array}{ccc} 0 & {\rm if} & \forall g \exists n (f(\bar g(n)) = 0)\\ 1 & {\rm if} & \exists g \forall n (f(\bar g(n))> 0)\end{array} \right.\]
\begin{lemma} The Suslin functional $\su$ is computable in $\ci$. \end{lemma}
\begin{proof}
We use that $^2E$ is computable in $\ci$. Given $f$, we let $T_f$ be the tree of finite sequences $s = (s_0 , \ldots , s_{n-1})$ such that we for all $m \leq  n$ have that $f(\langle s_0 , \ldots , s_{m-1}\rangle) = 0$. Then $\su(f) = 0$ if and only if $T_f$ is well founded. For all $f$, the subset of $T_f$ consisting of all sequences that cannot be extended to an infinite branch in $T_f$ can be defined  using an arithmetical inductive definition, and we then use $^2E$ to decide if this subset   is the whole tree $T_f$.\end{proof}
\subsection{Totality vs. partiality}\label{2.2}
In the original definition of higher order computability via Kleene's S1 - S9, all objects were assumed to be total. This can be considered to be a defect of S8, where the input $\lambda \xi.\{d\}(\xi , ----)$ has to be defined for all $\xi$ of the type in question in order to accept the termination of $\Psi(\lambda \xi.\{d\}(\xi, ----))$, even if $\Psi$ is defined in such a way that it only requires some values of the input functional. In the case of $\ci$, we only need $F$ to be total on the set of functions $f_\beta$ for $\beta \leq \alpha_F$ in order to identify $\ci(F)$.
\begin{remark}{\em A similar phenomenon takes place for Gandy's Superjump $\mathbb{S}$, introduced in \cite{Gandy}. The superjump is defined by
\[\mathbb{S}(F,e) = \left\{\begin{array}{ccc} 1 & {\rm if} & \{e\}(F,e) \downarrow\\ 0 & {\rm if} & \{e\}(F,e) \uparrow \end{array}\right.\]  where $\downarrow$ means that there is a value of the computation, while $\uparrow$ means the converse. In order to find the value of $\mathbb{S}(F,e)$ we only need to know $F$ restricted to the set of $f$ computable in $F$, the so called \emph{1-section} of  $F$. This was  used by Harrington \cite{H72} in an essential way when he classified the computational strength of $\mathbb{S}$, and was also important in Hartley's \cite{Hartley} analysis of the countably based functionals ( See Section \ref{3.} for a further discussion). We will show that loosening up the requirement of totality of the input functional to $\ci$ does not add to the computational strength of $\ci$. This is as it is for $\mathbb{S}$, but not, for instance, as for computations with continuous inputs in general. Then we add considerable strength by relaxing on S8, see e.g.  \cite[Sections 6.4 and 8.5]{LN} for results and further references.}\end{remark}
\begin{definition}\label{Part}
{\em We write $\{e\}_t(\ci,\vec F, \vec f , \vec a ) = b$ if $\{e\}(\ci,\vec F, \vec f , \vec a) = b$ according to the original definition, while we write $\{e\}_p(\ci ,\vec F,  \vec f , \vec a) = b$ if we interpret S8 according to the following extension of $\ci$ to partial $F:\Ca \rightarrow \Ca$. We will not accept non-total inputs to $F$, and for each $f \in \Ca$ we either have that $F(f) \in \Ca$ or totally undefined. We stick to the notation from Section \ref{2.1}:
\begin{itemize}
\item[i)] By recursion on $\beta$, $f^F_\beta$ is {\em defined }  if $\beta = 0$ or $\beta > 0$ and both $f^F_\gamma$ and $F(f^F_\gamma)$ are defined for all $\gamma < \beta$.
\item[ii)] $\ci(F)$ is \emph{defined }if there is an ordinal $\alpha$ such that $f^F_{\alpha +1 }$ is defined and $f^F_\alpha = f^F_{\alpha + 1}$
\item[iii)] If $\ci(F)$ is defined, and $\alpha$ is  as in ii), $\ci(F)(n) = f^F_\alpha(n)$ for each $n \in \N$.
\end{itemize}}\end{definition}
When the context is clear, we will talk about $t$-computations and $p$-computations.
\begin{theorem}
There is a computable (in the sense of Turing)  function $\rho$ such that if $\{e\}_p(\ci,\vec F, \vec f , \vec a) = b$, then $\{\rho(e)\}_t(\ci , \vec F, \vec f , \vec a) = b$.
\end{theorem}
\begin{proof}
We use the recursion theorem to define $\rho$, and define it by cases according to S1 - S9. It is obvious what to do in all cases except application of $\ci$. The final correctness proof will, of course, be by induction on the complexity of the $\{e\}_p$-computation (we will define the \emph{rank} or \emph{norm} of a terminating computation formally below, definitions that do not rely on the correctness of this theorem), but as is common for this kind of argument, we assume that $\rho$ does the job on all subcomputations, and we define $\rho$ by self-reference. 
\newline
So assume that \[\{e\}_p(\ci , \vec F, \vec f , b , \vec a) = \ci(F_{G_p})(b)\] where $G_p(f) = \{d\}_p(\ci , \vec F, f , \vec f ,\vec  a)$, and that the recursion terminates
 as defined   in Definition \ref{Part}.  Assume further, as an induction hypothesis, that we can replace $G_p$ with the, still partial,
 \[G_t(f) = \{\rho(d)\}_t(\ci ,\vec F, f , \vec f , \vec a).\]

\medskip

We assume familiarity with the concept of a \emph{prewellordering} $R$ on a domain $D \subseteq \N$. Since $^2E$ and $\su$ are $t$-computable in $\ci$, we also have that the set of prewellorderings will be $t$-computable in $\ci$.
\newline
If $R$ is a prewellordering on $D$, each element in the domain $D$ will have an ordinal rank, and we let $R_\beta$ be the elements in $D$ with ordinal rank below $\beta$. We will construct a total functional $H$ mapping prewellorderings to prewellorderings such that we can decide $b \in \ci(F_{G_t})$ from $\ci(H)$. The definition of $H(R)$ is as follows, observing that we only need $^2E$ when we know that $R$ is a prewellordering. We let   $f_\beta$  be as in the definition of $\ci(F_{G_t})$, and we \emph{identify} $f_\beta$ with $A_\beta = \{b \in \N : f_\beta(b) = 1\}$.
\medskip

- By $R$-recursion, compare $R_\alpha$ with $A_\alpha$ until we  either have disagreement or that $R_\alpha = A_\alpha$ with $F(A_\alpha) \subseteq A_\alpha$.
\smallskip

- In the first case,  $\alpha$ must be a successor ordinal $\beta + 1$. We let $H(R)$ be $R$ restricted to $R_\beta  = A_\beta$, and then  end-extended with  $F(A_\beta) \setminus A_\beta$. In the other case  we let $H(R)$ be $R$ restricted to $R_\alpha$. 
\medskip

Since we in the computation of $H(R)$ only will ask for values $F(A_\beta)$ , our assumption shows that $H$ is total. $\ci(H)$ will be a prewellordering $R$, and we will have that it matches the prewellordering induced by $F_{G_t}$. We then have that
\[b \in \ci(F) \Leftrightarrow b \in dom(\ci(H)).\]
It is now a  matter of routine to define a suitable candidate for $\rho(\langle 8,d\rangle   )$  in a computable way from $d$ and an alleged index for $\rho$, so we may define a working $\rho$ by the classical recursion theorem.
\end{proof}

 From now on, if we write $\{e\}$, then we mean $\{e\}_p$.
\subsection{The norm of a computation and Gandy Selection}\label{2.3}
The advantage of using $p$-computations is that now all computation trees will be countable, and all computations will have a countable ordinal as rank. We give a direct definition of this rank. In order to simplify the readability we introduce the following as a convention: With the expression $\lambda (g,c).\{d\}(\ci, \vec F, g , \vec f , c , \vec a)$ we really mean the function \begin{center} ($\ast$) \;\;\;\;$F(g)(c) = \min\{ 1 , \{d\}(\ci, \vec F, c \what g , \vec f , \vec a)\}$.\end{center}
\begin{definition}{\em Let $C_{\ci}$ be the set of finite sequences $\langle e , \vec F, \vec f , \vec a\rangle$ such that for some $b$ we have 
\[\{e\}(\ci , \vec F, \vec f , \vec a) = b.\]
If $\langle e , \vec F, \vec f , \vec a\rangle \in C_{\ci}$ we define the \emph{norm} $||\langle e , \vec F, \vec f , \vec a \rangle ||$ by transfinite recursion as follows:
\begin{itemize}
\item[i)] If $e$ corresponds to S1 - S3 or S7, we let the norm be zero.
\item[ii)] If $\{e\}(\ci , \vec F, \vec f , \vec a) = \{e_1\}(\ci , \vec F, \vec f , \{e_2\}(\ci , \vec F, \vec f , \vec a) , \vec a)$, where $ \{e_2\}(\ci , \vec F, \vec f , \vec a) = c$, we let 
\[||\langle e , \vec F, \vec f , \vec a\rangle|| = \max\{||\langle e_2 , \vec F, \vec f , \vec a \rangle||, ||\langle e_1 , \vec F, \vec f , c , \vec a \rangle||\} + 1.\]
The cases  S6 and S9 are handled in a similar way, and are left for the reader. 
\item[iii)] If $\{e\}(\ci , \vec F , \vec f , \vec a) = F_1(g)$ where $g(b) = \{d\}(\ci , \vec F , \vec f , b , \vec a)$ we let \[||\langle e,\vec F , \vec f , \vec a \rangle|| = \sup \{||\langle d , \vec F , \vec f , b , \vec a\rangle|| + 1: b \in \N\}\]
\item[iv)] If $\{e\}(\ci , \vec F, \vec f , b,\vec a) = \ci(\lambda (g,c).\{d\}(\ci ,\vec F, g,\vec f , c , \vec a))(b)$,
we let $F$ be as in $(\ast)$, and we let $\alpha$ and $f_\beta$ for $\beta \leq \alpha$  be as in the definition of $\ci$.  By the assumption, $f_\beta$ is well defined and total for all $\beta \leq \alpha$, where $F(f_\alpha) \leq f_\alpha$. 
We let 
\[||\langle e , \vec F,\vec f , b, \vec a\rangle|| = \sup \{||\langle d , \vec F, c \what f_\beta , \vec f  , \vec a \rangle|| + 1 : \beta \leq \alpha \wedge c \in \N\}.\]
\end{itemize}
If $\langle e , \vec F, \vec f , \vec a\rangle \not \in C_{\ci}$ we let $||\langle e , \vec F, \vec f , \vec a \rangle || = \aleph_1$, the first uncountable ordinal.
}\end{definition}
\begin{lemma}[Stage Comparison]\hspace*{2mm}\label{SC}
\newline
There is a a partial functional $P$ in two variables, $p$-computable in $\ci$, such that 
\begin{itemize}
\item[i)] $P(\langle e ,\vec F, \vec f , \vec a\rangle,\langle d ,\vec G,  \vec g , \vec c\rangle)$ terminates if at least one of $\langle e , \vec F, \vec f , \vec a\rangle$ and $\langle d ,\vec G,  \vec g , \vec c\rangle$ is in $C_{\ci}$ and then
\item[ii)] $P(\langle e , \vec F, \vec f , \vec a\rangle,\langle d ,\vec G,  \vec g , \vec c\rangle) = 1$ if $||\langle e ,\vec F,  \vec f , \vec a\rangle|| \leq ||\langle d , \vec G, \vec g , \vec c\rangle||$
\item[iii)] $P(\langle e , \vec F, \vec f , \vec a\rangle,\langle d , \vec G, \vec g , \vec c\rangle) = 0$ if $||\langle d , \vec G, \vec g , \vec c\rangle|| < ||\langle e , \vec F, \vec f , \vec a\rangle||$.

\end{itemize}\end{lemma}
\begin{proof}We use the recursion theorem to construct $P$, and the definition is split into 81 cases, according to the schemes corresponding to $e$ and $d$. S8 splits into two cases, S8.2 and S8.3 for applications of $F_1$ and $\ci$, while S5 is redundant and left out. This is why we have $9 \times 9$ cases. Strictly speaking there are 100 cases, because we must say what $P$ does in cases where one or both indices do not correspond to Kleene-indices at all, but we leave these trivial cases for the reader. Fortunately, many other cases are trivial as well, in particular those where one of the indices $e$ or $d$ represents a basic computation S1 - S3 or S6. Moreover, all cases not involving S8.3 are covered by the literature, see e.g. \cite{Gandy}.
\smallskip

 We will give the details for three cases (S4 , S8.3), (S8.2 , S8.3)  and (S8.3 , S8.3). The remaining cases follow by similar, or even simpler,  arguments. As is normal practise for this kind of construction/proof we define $P$ by self reference, assuming for each case, as an induction hypothesis, that $P$ works for the immediate subcomputations. 

\medskip

\noindent Case (S4 , S8.3): Let \[\{e\}(\ci, \vec F , \vec f , \vec a) = \{e_1\}(\ci , \vec F, \vec f , \{e_2\}(\ci , \vec F, \vec f , \vec a) , \vec a)\] and let \[\{d\}(\ci , \vec G,\vec g , b,\vec c) = \ci((G)(b),\]
where $G = \lambda(g,c).\{d_1\}(\ci , \vec G , c \what g,\vec g  , \vec c))$.

Let  $g_\alpha$ be element $\alpha$ in the sequence inductively defined from $G$.  We now consider the following induction, that can easily be formalised via an inductive definition:
\smallskip

Use $P$ to compare $||\langle e_1 , \vec F , \vec f , \vec a\rangle||$ with the ranks needed to compute $g_0$, $g_1$ , $\ldots$ until the first is bounded in norm by one of the latter computations or until the latter induction terminates.
\smallskip

In the first case, let $c = \{e_2\}(\ci , \vec F,\vec f , \vec a)$ and start over again, now comparing the computations involved in computing the $g_\alpha$'s with $||\langle e_1 , \vec F,\vec f , c , \vec c\rangle||$. 
\smallskip

If $||\langle e , \vec F, \vec f , \vec a\rangle|| \leq ||\langle d , \vec G, \vec g , \vec c\rangle||$, this will be verified through the two inductions, and the composition will terminate at least as fast as the induction. If $||\langle d , \vec G, \vec g , \vec c\rangle|| < ||\langle e ,\vec F,  \vec f , \vec a\rangle||$, at least one of the two inductions will result in the full induction induced by $G$, and we can deduce that this terminates faster than the composition. 
\smallskip

\noindent Case (S8.2 , S8.3): Let \[\{e\}(\ci , \vec F,\vec f , \vec a) = F_1(f)\] where \[f(a) = \{e_1\}(\ci , \vec F , \vec f , a , \vec a)\] and let \[\{d\}(\ci , \vec G,\vec g , b,\vec c) = \ci(G)(b)\] where $G$ is as in the previous case.

As in the previous case, we simulate the induction in the second part while, at each step, comparing the length of the computations needed with those of each $ \{e_1\}(\ci , \vec F , \vec f , a , \vec a)$. We use $^2E$ in doing this. If we for each $a$ reach a step in the induction where we need a computation that dominates the computation of  $ \{e_1\}(\ci , \vec F , \vec f , a , \vec a)$, we know that the left hand side will terminate at most with the same rank as the right hand side. If we are able to complete the induction on the right hand side before termination of all sub-computations on the left hand side, we know that the right hand side terminates first.  This stepwise comparison until the value of $P$ is settled can be expressed as an inductive definition.
\smallskip

\noindent Case (S8.3 , S8.3):  Let \[\{e\}(\ci ,\vec F, \vec f , a , \vec a) = \ci(\lambda (a',f').\{e_1\}(\ci , \vec F, a' \what f',\vec f, \vec a))(a)\] and let \[\{d\}(\ci ,\vec G, \vec g , b , \vec b) = \ci(\lambda (b',g').\{d_1\}(\ci , \vec G , b' \what g',\vec g, \vec b))(b).\] 
Notice that the norms  of these computations will be independent of the choices of $a$ and $b$. Let $F$ and $G$ be the partial functionals involved in these inductions, where at least one is sufficiently total for the induction to terminate. We now describe a simultaneous inductive definition of two increasing sequences $f_\alpha$ and $g_\beta$ of elements of $\Ca$, where we use $^2E$ and $P$ to make all the comparisons involved:
\begin{itemize}
\item[*] Let $f_0 = g_0$ be the constant zero.
\item[*] Assume that $f_0 , \ldots , f_\alpha$ and $g_0 , \ldots , g_{\beta}$ are constructed.
\item[*]
Consider all computations involved in computing all $f_\delta(a)$ for $\delta \leq \alpha$ and in computing $F(f_\alpha)(a)$ for all $a$, and then consider all computations involved in computing all $g_\gamma(b)$ for $\gamma \leq \beta$ and in computing $G(g_\beta)(b)$ for all $b$.
\item[*]If the norm of each computation in the first set is bounded by the norm of some computation in the second set, we add $f_{\alpha + 1} = \max\{f_\alpha , F(f_\alpha)\}$ and keep $g_0 , \ldots g_{\beta}$.
\item[*]
On the other hand, if there is one computation in the first set whose norm strictly bounds all norms of the computations in the other set, we add $g_{\beta + 1} = \max\{g_\beta , G(g_\beta)\}$ and keep $f_0 , \ldots , f_\alpha$.
\item[*]
At least one of these two inductions will terminate through this process, and when it does, we know which one will terminate with lowest ordinal norm.
\end{itemize}
We leave the formal definition of this inductive definition  for the reader.
\end{proof}

\bigskip

\begin{theorem}[Gandy Selection]
There is a $p$-computable \emph{selection operator} $\nu$ such that for all $e$, $\vec F$, $\vec f$ and $\vec a$ we have 
\[\exists n \{e\}(\ci , \vec F,  \vec f , n , \vec a) \!\!\downarrow\;\; \Rightarrow\; \{e\}(\ci , \vec F, \vec f , \nu(e, \vec F, \vec f , \vec a) , \vec a)\!\!\downarrow.\]
\end{theorem}
\begin{proof}
This is a soft consequence of Lemma \ref{SC}, with an argument well known in the literature, see e.g.\ \cite[Theorem 3.1.6]{JEF}, \cite[Theorem 3]{Johan}, \cite[Theorem X.4.1]{Sacks} or the original \cite{Gandy}.

\end{proof}

For many of the inductive definitions used, we add at most one new element to the inductively defined set at each stage. Such definitions can be defined by functionals $F$ of pure type 2, identifying $2^\N$ with the power set of $\N$ via characteristic functions:
\begin{definition}{\em Let $G:2^\N \rightarrow \N$ and let $H_G: 2^\N \rightarrow 2^\N$ be defined by
\[H_G(A) = A \cup \{G(A)\}.\]
An inductive definition $F$ is \emph{single valued} if it is in the form $H_G$. We let $\ci_0(G) = \ci(H_G)$
}\end{definition}
\begin{lemma} The functionals $\ci$ and $\ci_0$ are computationally equivalent modulo $^2E$.
\end{lemma}
\begin{proof}$\ci_0$ is trivially, and outright, computable in $\ci$.
In order to prove the other direction, we let $F:2^\N \rightarrow 2^N$ be given, and we will construct a single valued $G$ that \emph{simulates} $F$. We assume that $F$ is nontrivial, i.e. that $F(\emptyset) \neq \emptyset$. We let $G$ operate on sets $B$ of finite binary sequences, and we totally order these sequences using the standard lexicographical ordering by first comparing the first place where two sequences are different, and if this does not help, by length. This is not a well ordering, but $G$, as we define it, will only generate well ordered  sets of sequences. There will be three cases in the definition of $G(B)$:
\begin{enumerate}
\item $B$ has no maximal element $s$. Let $A$ be the set of $n$ such that $s(n) = 1$ for at least one $s \in B$. If $F(A) \subseteq A$, let $G(B)$ be the (sequence number of) the empty sequence. If not, let $n$ be the least number in $F(A) \setminus A$, and let $G(B) = s$ where $s$ is the binary sequence of length $n+1$ approximating the characteristic function of $A \cup F(A)$.
\item If there are elements  $s_1 < \cdots < s_k$ in $B$ so that $B_1 = \{s \in B : s < s_1\}$ has no maximal element and such that
\[B =  B_1 \cup \{s_1 , \ldots , s_k\},\] let $A$ be the set of $n$ such that $s(n) = 1$ for at least one $s$ in $B_1$. If the sequences $s_1 , \ldots , s_k$ do not approximate the characteristic function of $A \cup F(A)$, let $G(B) = 0$ (the value does not matter), while otherwise, we let $G(B) = s$ where $s$ is the least proper extension of $s_k$ that approximates $A \cup F(A)$.
\item Otherwise, let $G(B) = 0$.
\end{enumerate}
The induction induced by $G$ will, one step at the time, build up approximations to the characteristic functions of the sets appearing in the induction induced by $F$. If $\ci(F)$ uses $\alpha$ many steps, $\ci_0(G)$ will use $\omega \cdot \alpha$ many steps. Clearly $G$ is computable in $F$ and $^2E$, and the closure set of $F$ is arithmetical in the closure set of $G$. Thus $\ci$ is computable in $\ci_0$ and $^2E$. \end{proof}
\section{The companion of $\ci$}\label{Comp}
In this section we will analyse the computational power of $\ci$ in terms of set theory. Recall that a set $X$ is \emph{hereditarily countable} if the \emph{transitive closure} ${\rm trcl}(X)$ is countable. Hereditarily countable sets $X$ will have \emph{codes}, essentially structures $(D,R,A)$ where $D \subseteq \N$, $R$ is a binary relation on $D$, $A \subseteq D$ and $(D,R,A)$  is isomorphic to $({\rm trcl}(X), \in^{{\rm trcl}(X)},X)$. Such codes can further be coded as functions in $\N^\N$ in a natural way.
\begin{definition}{\em The \emph{companion} $\mathcal M$ of $\ci$ is defined as the set of sets $X$ with codes that are computable in $\ci$.}\end{definition}
\begin{remark}{\em The companion of other functionals are defined in analogy with this. For instance, the companion of $^2E$ will be $\LL_{\omega_1^{\rm CK}}$, the companion of $\su$ (the Suslin functional) is $\LL_\beta$ for the first recursively inaccessible ordinal $\beta$ while the companion of $\mathbb S$ (the Superjump)  is $\LL_\rho$ where $\rho$ is the first recursively Mahlo ordinal .}\end{remark}
\begin{lemma} There is a countable ordinal $\pi$  such that ${\mathcal M} =\LL_\pi$. \end{lemma}
\begin{proof}Since $\{e\}_p$ is absolute for $\LL$, we have that $\mathcal M$ is a transitive subset of $\LL$. $\LL$ will be closed under a certain map sending a code for an ordinal $\alpha$ to a code for $\LL_\alpha$ (this map is computable in $^2E$), so $\mathcal M$ will be an initial segment of $\LL$. \end{proof}
\begin{lemma}\label{lemma.compare} Let $F:\Ca \rightarrow \Ca$ be a partial functional computable in $\ci$ such that $\ci(F)$ is defined. Let $\alpha$ be the corresponding closure ordinal for $F$. Then $\alpha < \pi$. \end{lemma}
\begin{proof} For each $F$ there is an $F'$ computable in $F$ and $^2E$ such that $F'$ generates a prewellordering $R$ where $R_{\beta + 1} = F(R_\beta) \setminus R_\beta$ for each ordinal $\beta$. Then $\alpha$ will be the ordinal rank of the inductively definable prewellordering $R$ so $\alpha$ will be computable in $\ci$ whenever $F$ is computable in $\ci$. \end{proof}

The aim of this section is to find closure- and reflection-properties of $\LL_\pi$. 
Since $\su$ is computable in $\ci$ we have that the set of codes for hereditarily countable sets is computable in $\ci$. Given  codes $f_i$ for sets $X_i$, $i \in \N$, , we only need $^2E$ to unify the codes in the form of a code for $\{X_i : i \in \N\}$. Further, given codes $f_1 , \ldots , f_n$ for sets $X_1 , \ldots , X_n$, and a $\Delta_0$-formula $\Phi(x_1  \ldots , x_n)$, $^2E$ can decide the truth value of $\Phi(X_1 , \ldots , X_n)$. Finally, if $\Phi(x_1 , \ldots , x_n,y)$ is a $\Delta_0$-formula, $f_1 , \ldots , f_n$ are  codes  computable in $\ci$ for $X_1 , \ldots , X_n \in \LL_\pi$ and \[\LL_\pi \models \exists Y \Phi(X_1 , \ldots , X_n, Y)\]
then  we can use Gandy selection for $\ci$ to compute (an index for) a code $g$ for a set $Y$ such that $\Phi(X_1 , \ldots , X_n , Y)$. This leads to a proof of
\begin{lemma} $\LL_\pi$ is an admissible structure. \end{lemma}

Let $\WO$ be the set of codes for countable ordinals. This is a $\Pi^1_1$- set, and it is easy to prove that the following sets are $\Pi^1_1$ as well:
\begin{enumerate}
\item The set of $f \in \WO$ that codes $\omega_1^{\rm CK}$.
\item The set of $f \in \WO$ that codes the first recursively inaccessible ordinal.
\item The set of $f \in \WO$ that codes the first recursively Mahlo ordinal.
\end{enumerate}
We say that these ordinals are \emph{$\Pi^1_1$-characterisable}. Many ordinals of distinction are $\Pi^1_1$-characterisable, for instance all \emph{clockable} ordinals in the sense of infinite time Turing machines (\cite{HL}), see Welch \cite{Welch.Turing} for a survey and further references on such machines.
\begin{definition}{\em Let ${\bf P}$ be a class of ordinals. We say that ${\bf P}$ is \emph{$\ci$-decidable} if there is an $\ci$-computable function $\Delta:\N^\N  \rightarrow \N$ such that $\Delta(f) = 0$ if and only if $f$ codes an ordinal $\alpha$ and ${\bf P}(\alpha)$ holds.}\end{definition}
 That $\pi$ is not $\Pi^1_1$-characterisable follows from the following much stronger:
 \begin{theorem}\label{thm.improved}Let $\bf P$ be a property on ordinals that is $\ci$ decidable and such that ${\bf P}(\pi)$.
 Let $X \subset \pi$ be closed, unbounded and $\Sigma_1$ over $\LL_\pi$. Then there is an $\alpha \in X$ such that ${\bf P}(\alpha)$.
 \end{theorem}
 \begin{proof}
 We can code a partially enumerated set $\{f_d : d \in D\}$ of functions as the set of pairs $\langle d , \overline f_d(n )\rangle$ where $d \in D$ and $n \in \N$. The idea is to construct an inductive definition $\Gamma$ that is computable in $\ci$ and such that $\Gamma$ generates a code for an ordinal both in $X$ and satisfying $\bf P$. $\Gamma$ will not be total, but sufficiently total for the induction to terminate. In defining $\Gamma$ as computable in $\ci$, we use that the Suslin functional $\su$ is computable in $\ci$.
 We define  $\Gamma(R)$ as follows:
 \begin{itemize} 
 \item If $R$ does not code an enumerated set $\{f_d : d \in D\}$, we let $\Gamma(R) = R$. Note that the empty set codes the empty set of functions.
 \item Assume that $R$ codes $\{f_d : d \in D\}$. If $f_d \not \in \WO$ for some $d \in D$, let $\Gamma(R) = R$.
 \item Assume now that $f_d \in \WO$ codes $\alpha_d$ or all $d \in D$, and use $^2E$ to compute a code $g$ for the least upper bound $\alpha$ of $\{\alpha_d : d \in D\}$. If each $\alpha_d$ are in $X$, then $\alpha \in X$ since $X$ is closed. If ${\bf P}(\alpha)$ we let $\Gamma(R) = R$. This is where we want the induction to close.
 \item Otherwise, we apply Gandy selection for $\ci$ and search for an index $e$ for a code $g$ of an ordinal $\beta > \alpha$ such that $\beta \in X$. We then let \[\Gamma(R) =  R \cup \{\langle e,\overline g(n)\rangle : n \in \N\}.\]
 \end{itemize}
 If $\alpha < \pi$, we can use the recursion theorem for $\ci$ to see that $L_\pi$ is closed under the $\alpha$-iteration of $\Gamma$, and that $\Gamma$ generates codes for an increasing sequence of ordinals $\gamma_\beta$ for $\beta < \alpha$. Since we always use an index $e$ for an ordinal larger than those appearing at earlier stages, we do not risk to mix up codes for different ordinals. Since $X$ is closed, all ordinals  obtained during this iteration will be codes for ordinals $\gamma_\beta \in X$.  Since ${\bf P}(\pi)$ and this induction will stop when we hit a $\gamma_\beta$ with ${\bf P}(\gamma_\beta)$, and since by Lemma \ref{lemma.compare} no such induction will stop at $\pi$, there must be an ordinal $\gamma_\beta < \pi$ such that ${\bf P}(\gamma_\beta)$.
 
 \end{proof}
 
 \begin{corollary}\label{cor.car.1} The closure ordinal $\pi$ of $\ci$ is not $\Pi^1_1$-characterisable.\end{corollary}
 
 We also have

 \begin{corollary} The closure-ordinal $\pi$ of $\ci$ is recursively Mahlo.\end{corollary}
 
 \begin{proof} We have to prove that if $X \subseteq \pi$ is $\pi$-computable, closed and unbounded, then $X$ contains an admissible ordinal. Since the class of countable, admissible ordinals is $\ci$-decidable, this is a direct consequence of Theorem \ref{thm.improved}. \end{proof}
 Since being recursively Mahlo and other even stronger closure properties are also $\ci$-decidable, we may extend this argument in order to prove that $\pi$ satisfy these stronger properties, and that every closed unbounded subset of $\pi$ that are $\Sigma_1$ over $\LL_\pi$ also contain elements satisfying these stronger properties. We will not pursue this further here.
 \smallskip
 
  We will now consider an alternative way of expressing that $\pi$ must be a ``large" countable ordinal. What is ``large" is of course subject to the perspective one is taking.

\begin{definition}{\em An ordinal $\gamma$ is \emph{reflecting} if for all formulas $\Phi(x_1 , \ldots , x_n)$ and elements $X_1 , \ldots ,X_n$ in $\LL(\gamma)$, 
\[\LL_\gamma \models \Phi(X_1 , \ldots , X_n) \Rightarrow \exists \beta < \gamma [L_\beta \models \Phi(X_1 , \ldots , X_n).]\]}\end{definition}
Note that if $\gamma$ is reflecting, then $\gamma$ is admissible and recursively inaccessible.   
\begin{corollary} The closure ordinal $\pi$ of $\ci$ is reflecting.\end{corollary}
\begin{proof}  If $\alpha < \beta$, $X_1 , \ldots , X_n$ are in $\LL_\alpha$ and $\LL_\pi \models \Phi(X_1 , \ldots , X_n)$, then the set of $\gamma > \alpha$ such that $\LL_\gamma \models \Phi(X_1 , \ldots , X_n)$ is $\ci$-decidable, contains $\pi$ and thus, by Theorem \ref{thm.improved}, contains an ordinal $\beta$ with $\alpha < \beta < \pi$.
\end{proof}
\noindent $\pi$ will not be the least reflecting ordinal:
\begin{corollary} Let $\pi$ be the closure  ordinal of $\ci$. If $\alpha < \pi$, then  there is a reflecting ordinal $\gamma$ with $\alpha < \gamma < \pi$. \end{corollary}
\begin{proof}

This is also a consequence of Theorem \ref{thm.improved}, since being reflecting is $\ci$-decidable. Indeed, if $X \subseteq \pi$ is closed, unbounded and $\Sigma_1$ over $\LL_\pi$, then $X$ contains arbitrarily large reflecting ordinals.

\end{proof}
\begin{remark}{\em These results do of course not imply that $\LL_\pi$ has an elementary substructure, or even a substructure satisfying the same first order sentences. The \emph{theory} of $\LL_\pi$ is not $\ci$-decidable, so there is no way to unify these arguments to all formulas simultaneously.}\end{remark}

\section{Classes of  functionals of type 3}\label{3.}
\subsection{Motivation}\label{3.1}
We introduced the functional $\ci$ in Section \ref{2.} and illustrated its computational strength through an analysis of the companion in Section \ref{Comp}. In this section we will give an analysis of computations relative to  $\ci$ resembling an \emph{operational semantics}. Our approach is inspired by the success-story of using \emph{nested sequential procedures} for modelling functionals definable in LCF (Scott, \cite{scott}) or equivalently in PCF (Plotkin, \cite{plotkin}) from objects of type 1. For an introduction to nested sequential procedures, see \cite[Chapter 6]{LN}. 
\smallskip

Since we are only concerned with functionals of type $\leq 3$ in this paper, we can forget the qualifier `nested', while we will add the qualifier `hyper' in order to deal with discontinuity. We will aim for more and more restricted concepts of \emph{hyper-sequential procedures} until we find a characterisation of the functionals of type 3 that are computable in $\ci$ and some functional of type 2. The gain will mainly be that we obtain a more civilised, and less ad hoc, way of expressing relative computability for certain functionals of type 3 than when we refer to the Kleene schemes directly. We will use this to give mathematical support to the informal claim that if realisers of classical theorems based on compactness arguments are computable in $\ci$, then the full power of $\ci$ is required.

\medskip

A functional $\Phi$ of type 3 is \emph{normal} if $^3E$ is computable in $\Phi$, where 
\[^3E(F) = \left \{ \begin{array}{ccc} 0 & {\rm if} & \forall f \in \N^\N (F(f) = 0) \\ 1&{\rm if}&\exists f \in \N^\N(F(f) > 0)\end{array}\right . , \]
and where $F$ is assumed to be total.

The set  of functionals of type 3 that are neither normal nor computable in type 2 objects is mainly unexplored with respect to computability-theoretical properties. The classical object of this kind is Gandy's Superjump $\mathbb S$.  $\mathbb S$ is of course a natural functional in the context of higher  order computability theory. Recently, examples that are natural from other perspectives have emerged. In \cite[\S 3]{P1} we introduced classes of realisers  $\Theta$ for the general Heine-Borel theorem and a weaker class of functionals $\Lambda$ that compute realisers for the Vitali Covering theorem. In \cite{N18} we  also considered    functionals $\Xi$ that serve as realisers for the Lindel\"of Lemma for Baire Space. In this paper we introduced $\ci$, which, under the name IND, was proved in \cite{N18} to compute Lindel\"of realisers $\Xi$. This plethora of elements in a so far unexplored class of functionals justifies a more coherent study of this class. We will return to some of these functionals in Section \ref{7.2}.
\medskip

 Hartley \cite{Hartley} investigated the fully typed hierarchy of hereditarily countably based functionals, based on a definition due to Stan Wainer, and obtained some general results. For instance, he proved that if we assume the Continuum hypothesis together with  ZFC, $\Phi$ is countably based if and only if $^3E$ is not computable in $\Phi$ and any functional of type 2. 
\smallskip

The original definition of the countably based functionals is by a generalisation of the definition of the \emph{continuous} functionals e.g. as based on \emph{domain theory}, see \cite[Chapter 10]{LN} for a recent introduction. In this paper, we will only be interested in objects of types 0, 1, 2 and 3, and we define the countably based functionals for these cases, suiting our own purposes:
\begin{definition}{\em  All integers are countably based. Moreover
\begin{enumerate}
\item All total functions $f:\N \rightarrow \N$ are countably based.
\item All partial functionals $F$ mapping a subset of $\N^\N$ to $\N$ are countably based.
\item Let $\Phi$ be a partial functional taking countably based functionals of type 2 as arguments and yielding integers as values. $\Phi$ is countably based if we for each $F$ and $n$ such that $\Phi(F) = n$ find a countable set $A \subseteq \N^\N$ such that $F$ is total on $A$ and such that for all $G$ of type 2, if $G$ is total on $A$ and agrees with $F$ on $A$, then $\Phi(G) = n$.
\end{enumerate}
In (3), a \emph{base element} for $\Phi$ will be a countable set $A$ together with the restriction of an $F$ to $A$ with the property described.}\end{definition}
$^3E$ will not be countably based, since in order to know that $^3E(O^2) = 0$  we need to know that $O^2(f) = 0$ for all $f \in \N^\N$.

\medskip

 One problem with the countably based functionals $\Phi$ is that the  base elements of $\Phi$ are not well structured as individual sets, and a suitable class of base elements for $\Phi$ may not be well structured as a class. Much of the way of thinking inherited from the computability theory of the continuous functionals is useless. The aim of this section is to introduce a more restricted class, the hyper-sequential functionals, where we have added some further structure. Examples of hyper-sequential functionals will be the Superjump and $\ci$. However, the first concept we introduce will be too general for our purpose, for instance, all functionals of type 3 computable using an \emph{infinite time Turing machine} the way suggested by Welch \cite{Welch} will be hyper-sequential.
 
 \subsection{Hyper-sequential functionals}\label{3.2}
 \subsubsection{The definitions}
 In this section we will define what we mean with a hyper-sequential procedure. A transfinite calculation using a functional $F$ as an oracle can be viewed as a sequence of queries of the form  ``what is $F(f)?$", where the next query will depend on the answer to the previous ones. We will capture such deterministic procedures with our concepts defined in \ref{def.proc}.  Our aim will be to isolate the procedures that will correspond to computations relative to $\ci$, and in order to fully capture those , our calculations also must contain some \emph{documentation}. In a computation $\{e\}(\ci , F , \vec f , \vec a)$, there may be subcomputations with extra arguments $g$ or $b$ of type 1 or 0. Our abstract calculations will contain a LOG of functions $g$, and the use of this LOG will be to show that procedures corresponding to Kleene-computations in $\ci$ are definable at the level of $\Pi^1_1$. This will be made precise later. In order to formally describe this LOG we take the liberty to add an extra element $\ast$ to $\N$, and to claim that objects involving this $\ast$ will be of a certain complexity, for instance $\Pi^1_1$, without going to the trouble of coding.
 \begin{definition}{\em \hspace*{2mm}\label{def.proc}
 \begin{itemize}
 \item[a)] A \emph{string} is a sequence $\{(f_\beta , a_\beta)\}_{\beta < \alpha}$ where $\alpha$ is a countable ordinal, each $f_\beta \in \N^\N$ and each $a_\beta \in \N\cup\{\ast\}$. We call $f_\beta$ a \emph{query}, and sometimes writes it as $F(f_\beta) = ?$. 

   \item[b)] A \emph{hyper-sequential procedure} is a set $\Omega$ of strings where each string will be given an integer value, and  such that whenever  $\{(f_\beta,a_\beta)\}_{\beta < \alpha}$ and $ \{(f'_\beta,a'_\beta)\}_{\beta < \alpha'}$ are in $\Omega$  they are either equal or there is a  $\beta < \min\{\alpha,\alpha'\}$ such that $f_\beta = f'_\beta$, $a_\beta \neq a'_\beta$, $a_\beta \neq \ast$, $a'_\beta \neq \ast$ and $(f_\gamma,a_\gamma) = (f'_\gamma , a'_\gamma)$ for all $\gamma < \beta$. 
 Formally,  will let $\Omega$ be a set of pairs $(t,b)$ where $t$ is a string and $b$ is the associated value.
 \item[c)] If $\{(f_\beta , a_\beta)\}_{\beta < \alpha}$ is a string and $F \in Tp(2)$, we say that the string \emph{matches} $F$ if $F(f_\beta) = a_\beta$ for all $\beta < \alpha$ with $a_\beta \in \N$. 
 \item[d)] If a string $t$ is in a a hyper-sequential procedure $\Omega$, has a value $a$ and matches $F$, we call $t$ a {\em calculation}, calculating $\Omega(F) = a$.
  \item[e)] If $\Omega$ is a hyper-sequential procedure, then $\Omega$ \emph{defines} (or \emph{computes}) the partial functional $\Phi(F) = \Omega(F)$ of type 3. When $\Phi(F)$ is defined, the calculation of $\Omega(F)$ will be unique.
 \item[f)] A \emph{total} functional of type 3 is hyper-sequential if there is a hyper-sequential procedure that defines it. 
 \item[g)] If $t = \{(f_\beta , a_\beta)\}_{\beta < \alpha}$ is a calculation, and $a_\beta = \ast$ we say that $\beta$ is in the LOG of $t$.
 
 \end{itemize}
 }\end{definition}
 \smallskip
We will  from now on  use the words \emph{procedure} and \emph{sequential} in the meaning of hyper-sequential procedure and hyper-sequential.
 \begin{remark}{\em A procedure can be viewed as a \emph{strategy} for a transfinite \emph{game} where Player I, the computing device, plays \emph{queries} and Player 2, the input, \emph{answers} each query using $F$. In some matches of the games, corresponding to the calculations, Player 1 wins in  the sense of providing an output, while in other matches, Player II wins because it either stops after countably many steps without a value, or  it goes on through $\aleph_1$ many steps. We will discuss this further when we consider procedures with more structure. If we then still use the picture of games with rules, the LOG will represent places where Player 1 will enter a sub-game following different rules, and the LOG will help the referee to verify that the whole match is played according to the general, nested,  rules of the game.
 
 Note that if $t$ is a string that is an initial segment of several calculations, then the next $f_\beta$ will be the same for all such extensions, and if $\beta$ is in the LOG of one of them, it will be in the LOG of all extensions.
 
 }\end{remark}

 \begin{definition}{\em
 Let $\{(f_\beta , a_\beta)\}_{\beta < \alpha}$ be a  string.

 A \emph{sub-string} is a sequence $\{(f_\gamma,a_\gamma)\}_{\gamma < \beta}$ for some $\beta \leq \alpha$.

 }\end{definition}
 
 We can concatenate strings in the usual way: If we for each ordinal $\gamma < \gamma_0$  have a string $\{(f_{\gamma,\beta} , a_{\gamma,\beta})\}_{\beta < \alpha_\gamma}$ we let the \emph{concatenation} $\{(f_\beta , a_\beta)\}_{\beta < \alpha}$ be defined by
 \begin{itemize}
 \item[-] $\alpha = \sum_{\gamma < \gamma_0}\alpha_\gamma$
 \item[-] If $\beta = \sum_{\gamma < \gamma_1}\alpha_\gamma + \beta_1$ where $\gamma_1 < \gamma_0$ and $\beta_1 < \alpha_{\gamma_1}$, then $(f_\beta , a_\beta) = (f_{\gamma_1 , \beta_1},a_{\gamma_1 , \beta_1})$.
 
 \end{itemize}
  
 \bigskip

 We will prove that the class of sequential functionals of type 3 is closed under Kleene-computability as defined through the schemes S1 - S9. To be more precise, we will prove that if $\vec \Phi = (\Phi_1 , \ldots , \Phi_n)$ consists of sequential functionals and \[\lambda F. \{e\}(\vec \Phi,F,\vec f , \vec a)\] is total, then it is itself sequential. To make this precise, we need to extend S8 to deal with general inputs of type 3. For the sake of notational simplicity, we assume that the arguments of our computations will be of the form as above, that we drop the scheme S6 of permutation and that we use an alternative indexing for scheme S8 so that we can read out from the index for which of the arguments in the list $\vec \Phi$ the oracle call is made. (Alternatively we could modify S6 to cater for permutations of the list of inputs of all four types.) We still leave out S5, primitive recursion, partly because it is redundant in the presence  of S9, and partly because it can be handled in analogy   to  composition S4. Thus we add the following scheme to Definition \ref{Kleene}, while replacing the one occurrence  of $\ci$ with a sequence $\vec \Phi$ of functionals of type 3:
 \begin{itemize}
 \item[S8] If $e = \langle 8,3,i,d\rangle $ then 
 \[\{e\}(\vec \Phi , \vec F , \vec f , \vec a) = \Phi_i(\lambda f.\{d\}(\vec \Phi , \vec F , f,\vec f , \vec a))\]
 \end{itemize}
 
In the original definition by Kleene, this is only supposed to make sense when $\{d\}(\vec \Phi , \vec F , f,\vec f , \vec a)$ terminates for all $f \in \N^\N$, but when we are working with countably based $\Phi_i$ we normally only require that a base element is a sub-function of $\lambda f. \{d\}(\vec \Phi, \vec F , \vec f , \vec a)$.
 
 \smallskip
 
 As we will see in the sequel, being sequential the way we define it here is quite general, and thus the fact that this class is closed under Kleene computability may be of restricted interest. However, we will later refer to the construction of procedures imbedded in the proof of Lemma \ref{lemma.hyp} in situations where we will show that much more restricted classes of functionals still are Kleene closed.
 \smallskip
 
 Since we, in this section, are primarily interested in functionals of type 3 computable in a given sequence of sequential functionals of the same type, we restrict the number of arguments of type 2 to one. We can do this because the number of type 2 arguments will not increase as we move down the paths of the computation tree. The number of arguments of type 0 and of type 1 may increase, so we need to consider arbitrarily long finite lists of such input arguments.
 \begin{definition}{\em \label{Def.3.3}
 Let $\vec \Phi = (\Phi_1 , \ldots , \Phi_n) $ be a sequence of sequential functionals defined from the procedures $\Omega_1 , \ldots , \Omega_n$. Let $F$ be of type 2 and let $\vec f$, $\vec a$ be finite sequences of objects of type 1 and 0 resp.
 Assume that $\{e\}(\vec \Phi , F , \vec f , \vec a) = b$. By recursion on the length of this computation we define the calculation $t_{e, \vec \Phi , F , \vec f , \vec a}$ with value $b$ as follows, where we use $\what$ to denote concatenation of strings (recall that \emph{calculations} are strings that, in the given context, have values) : 
 \begin{itemize}
 \item[-] If $e$ is an index for an initial computation, i.e.\ for S1, S2, S3 or S7, we let $t_{e, \vec \Phi , F , \vec f , \vec a}$ be the empty string, i.e.\ with $\alpha = 0$.
 \item[-]If \[\{e\}(\vec \Phi , F , \vec f , \vec a) = \{e_1\}(\vec \Phi , F , \vec f , \{e_2\}(\vec \Phi , F, \vec f , \vec a) , \vec a), \] let $c = \{e_2\}(\vec \Phi , F \vec f , \vec a) $. Let \[t_{e, \vec \Phi , F , \vec f , \vec a} = t_{e_2, \vec \Phi , F , \vec f , \vec a}\what t_{e_1, \vec \Phi , F , \vec f , c,\vec a}.\]
 \item[-] In the case of S9, we just use the calculation for the immediate subcomputation.
 \item[-] Let $\{e\}(\vec \Phi , F , \vec f , \vec a) = F(\lambda c.\{e_1\}(\vec \Phi , F , \vec f , c , \vec a))$. Let $f(c) = \{e_1\}(\vec \Phi , F , \vec f , c , \vec a)$. Then 
 \[t_{e, \vec \Phi , F , \vec f , \vec a} = t_{e_1, \vec \Phi , F , \vec f , 0,\vec a}\what t_{e_1, \vec \Phi , F , \vec f , 1,\vec a}\what \underbrace{\cdots \cdots }_\omega\what(f,F(f)).\]
 \item[-] Let $\{e\}(\vec \Phi , F , \vec f , \vec a) = \Phi_i(\lambda g.\{e_1\}(\vec \Phi , F , g , \vec f , \vec a))$. Let $H(g) = \{e_1\}(\vec \Phi , F , g , \vec f , \vec a)$ and let $\{(g_\beta , b_\beta)\}_{\beta < \alpha}$ be the calculation  in $\Omega_i$  that is  matching $H$.
 \newline
 We let $t_{e, \vec \Phi , F , \vec f , \vec a}$ be the concatenation of \[\{(g_\beta,\ast) \what t_{e_1, \vec \Phi , F , g_\beta, \vec f , \vec a}\}_{\beta < \alpha}.\] \end{itemize}
 This ends the definition.
  
 }\end{definition}
 \begin{remark}\label{remark.dummy}{\em   We inserted the pairs $(g_\beta , \ast)$ in the LOG in order to remind us of the fact that we  at that stage are simulating a subcomputation with an extra argument $g_\beta$. We need the information about this extra argument  in order to say that a string is `correct', in a sense made precise later. }\end{remark}
 \begin{lemma}\label{lemma.hyp} Let $e$,$\vec \Phi$, $\vec f$ and $\vec a$ be fixed as in Defintion \ref{Def.3.3}. Then the set  \[\{(t_{e, \vec \Phi , F , \vec f , \vec a},b) : \{e\}(\vec \Phi,F, \vec f , \vec a)= b\}\] will be a procedure. \end{lemma}
\begin{proof}
 Assume that both $\{e\}(\vec \Phi , F , \vec f , \vec a)$ and $\{e\}(\vec \Phi , G , \vec f , \vec a)$ terminate. 
 We prove by induction on the ordinal ranks of the computations that the corresponding calculations  satisfy Definition \ref{3.2} b). The proof is split into cases corresponding to the Kleene schemes.
\bigskip

If $e$ is an index for an initial computation, the claim is trivial, and the induction step is trivial in the case of application of S9.
\smallskip

 Let $e$ be an index for composition, and let $e_1$ and $e_2$ be as in the construction.
 If the calculations for $\{e_2\}(\vec \Phi , F , \vec f , \vec a)$ and $\{e_2\}(\vec \Phi, G , \vec f , \vec a)$ are different, then by the induction hypothesis they split at a first point, and there the $f$-parts are the same while the $a$-parts differ.  Since these calculations are initial segments of the calculations under consideration, the concatenated calculations   also   satisfy the definition. 

 If the calculations for the $e_2$-computation are equal, then, by the indiction hypothesis, the values are the same, $c$, and then our conclusion follows from the induction hypothesis for $\{e_1\}(\vec \Phi , F/G, \vec f , \vec a)$.
\smallskip

 Application of $F/G$: In  this case, we construct calculations as the concatenation of $\omega + 1$ items, first the corresponding calculations for each $c \in \N$, and at the end, pairs $(f,F(f))$ and $(f',G(f'))$ respectively. If there is a least $c$ where the corresponding two calculations differ, the $f$-parts will agree while the $a$-parts will differ at a minimal location in these calculations, by the induction hypothesis. Then the $f$-parts will agree and the $a$-parts will differ at the corresponding minimal location in the concatenated calculation. If the two concatenations of the calculations inherited for each $c$ are equal, it follows from the induction hypothesis that the arguments $f$ and $f'$ for $F$ and $G$ resp. are equal, so at the top pair $(f,F(f))$  and $(f',G(f'))$ we will have that the query parts are equal. 
\smallskip

Application of $\Phi_i$: Let $\{(g_\beta , b_\beta)\}_{\beta < \alpha}$ and $\{(g'_\beta,b'_\beta\}_{\beta < \alpha'}$ be the two  calculations  in $\Omega_i$  matching the corresponding functionals $H$ and $H'$ as in the definition in this case. First we see that if the two concatenated calculations agree as far as they both go, we can use the induction hypothesis, sub-induction on $\beta < \min\{\alpha,\alpha'\}$ and the fact that $\Omega_i$ is a procedure to show that $g_\beta = g'_\beta$ and that $H(g_\beta) = H'(g'_\beta)$ for all $\beta$. Since $\Omega_i$ is a procedure, it follows that $\alpha = \alpha'$, that the two concatenated calculations are equal and that the values are the same.
 
 If the two concatenated calculations differ, there will be a least $\beta < \min\{\alpha,\alpha'\}$ such that they differ in the sections computing $H(g_\beta)$ and $H'(g'_\beta)$. Then $g_\beta = g'_\beta$, so by the induction hypothesis there is a least location in those sections where they differ, and there the $f$ parts are equal while the $a$-parts differ. So, the calculations constructed will satisfy the definition.
 \end{proof}
 \begin{theorem} \label{3.4} The class of hyper-sequential functionals of type 3 is closed under relative Kleene-computability.
 \end{theorem}
\begin{proof} Immediate from Lemma \ref{lemma.hyp}.\end{proof}
\begin{remark}{\em We can deduce, from the proof of Theorem \ref{3.4}, that all functionals of type 3 computable in functionals of lower types will be hyper-sequential. }\end{remark}
 \subsubsection{Mixed  types}\label{3.3}
 Some of the objects we are interested in are of types at  level $\leq 3$ that are not pure, $\ci$ is one prominent example. There are two natural ways to extend the concept of sequential functionals to objects of such types. One is to identify such types as the fixed points of computable retracts on the corresponding pure types, the retracts being explicitly definable as Kleene-computable where the schemes S5 and S9 are not needed. Then an object will be, by definition, sequential if the representation in the pure type is so.
 \newline
 The other alternative is to extend the intuition  of sequentiality to objects of these general types. A type like this will be of the form \[\sigma_1 , \ldots , \sigma_n \rightarrow \N,\]
 where each $\sigma_i$ has level $\leq 2$. Thus a calculation will be a well-ordered set of queries with answers where each query is of the form $F_i(\vec f) = ?$ for some $i$, varying with the query. Each $\vec f$ will consist of functions and/or integers, and the functions may be of one or several number variables. To keep track of all this in its full generality will require some heavy notation, but there will be no genuine mathematical problems. Given this, we can define what we mean with a \emph{procedure} adjusted to each type, and then the sequential objects of that type.
It is obvious that the two approaches are equivalent, but not being pressed, we prefer to omit all details. In some of our examples, we will use the latter, intuitive approach.
 \subsubsection{Examples}\label{Ssec3.4}
 Our first example is what motivated us to isolate the concept of hyper-sequential functionals:
 \begin{theorem}The functional $\ci$ is hyper-sequential.
 \end{theorem}
\begin{proof}
 Let $F:2^\N \rightarrow 2^\N$, $F'(a \what f) = F(f)(a)$ and let $f_0^F$ be the constant 0. We find $f^F_1 = F(f^F_0)$ through the $\omega$-series of queries $F'(a \what f_0^F) = ?$, then $f^F_2 = F(f^F_1) \cup f^F_1$ (identifying a characteristic function with the corresponding set) through the $\omega$-sequence of queries $F'(a \what f_1^F) = ?$ and so on. This is clearly a hyper-sequential procedure. \end{proof} 
 \medskip
 
 In his CiE-2019-paper \cite{Welch}, Philip Welch introduced infinite time Turing machines that can take functionals $F$ of type 2 as oracles. The idea is to have a special oracle tape, and whenever the oracle $F$ is called upon, we consider the oracle tape as the input information, and what the consequence of the oracle call will be will depend on the precise ITTM-model we are using. We have
\begin{theorem} Every ITTM-computable functional is sequential. \end{theorem}
We leave this theorem without a proof, since the proof is easy, but requires familiarity with the ITTM-model.

\medskip

Clearly, all sequential functionals are countably based. To what extent the converse is true is unknown, but we do have:

\begin{theorem}\label{3.9}
If the continuum hypothesis CH holds, all countably based total functionals will have extensions to partial functionals that are are sequential.
\end{theorem}
\begin{proof} We work within ZFC + CH. 
Let $\{f_\alpha\}_{\alpha < \aleph_1}$ be an enumeration of $\N^\N$. Let $\Phi$ be countably based and let $X$ be  set of base elements for $\Phi$. The elements of $X$ will be triples $(A,\phi,a)$ where $A \subseteq \N^\N$ is countable,  $\phi:A \rightarrow \N$ and $a \in \N$. The significance is that whenever $F$ extends $\phi$ to all of $\N^\N$, then $\Phi(F) = a$, and that for each $F$ there will be at least one $(A,\phi,a) \in X$ where $F$ is an extension of $\phi$.
\smallskip

The sequential procedure will then be to compute $F(f_\alpha)$ up to the first $\alpha_0$ where there is some $(A,\phi,a) \in X$ such that
\begin{itemize}
\item $A \subseteq \{f_\alpha : \alpha < \alpha_0\}$.
\item For $f_\alpha \in A$ we have that $F(f_\alpha) = \phi(f_\alpha)$.
\end{itemize}
We will have that $\Phi(F) = a$ independent of which $(A,\phi,a)$ we chose with this property.
\end{proof}
\subsection{Denotation procedures}\label{4.}
There is no reason to believe that the continuum hypothesis can be avoided in Theorem \ref{3.9}, but the theorem still suggests that the concept of hyper-sequential functional is too general to be of interest, and the intention is to investigate possible refinements of the concept. Now we will consider procedures that will include some extra information, a number or \emph{denotation} $d_\beta$ for each $\beta$ in the index ordinal of a calculation. In its full generality, this does not restrict the class of functionals definable from procedures, but it gives us a tool for discussing the complexity of them. Thus, in the theorems of this section, the constructions of the  procedures  with denotations used to prove them will be as important as the theorems themselves.

\begin{definition}\label{Def.4.1}{\em A \emph{denotation procedure} $\Omega$, \emph{d-procedure} for short,  will be a set $\Omega$ of \emph{  calculations with denotations}  \[( \{( f_\beta,a_\beta,d_\beta )\}_{\beta < \alpha},c)\] where each $a_\beta \in \N \cup \{\ast\}$ and \begin{enumerate}
\item The \emph{denotations} $d_\beta$ are in $\N$.
\item The corresponding set of calculations without the denotations is a procedure.
\item For each $( \{( f_\beta,a_\beta,d_\beta )\}_{\beta < \alpha},c) \in \Omega$, if $\beta < \gamma < \alpha$, then $d_\beta \neq d_\gamma$.

\end{enumerate}
 }\end{definition}
 By abuse of terminology, we will use $\Omega$ both for a d-procedure and for the corresponding procedure, making it clear in each case if we consider the denotations or not.
Clearly, all d-procedures will define sequential, partial functionals as well. In fact we have
 \begin{observation} By the axiom of choice, all procedures $\Omega$ can be extended to  d-procedures.\end{observation}
 We simply use the axiom of choice to select one enumeration of $\alpha$ for each calculation $(\{(f_\beta, a_\beta)\}_{\beta < \alpha} , c)  \in \Omega$ and use this to define the additional $d_\beta$s for each calculation. There is of course no extra knowledge to be harvested from this argument, but it illustrates a possibility that we have to bring under control in the d-procedures that we construct:
 \begin{definition}\label{def.delay}{\em 
 
 Let $\Omega$ be a d-procedure, let $\{(f_\beta , a_\beta , d_\beta)\}_{\beta < \alpha}$ be a calculation in $\Omega$ and let $\beta < \alpha$. The \emph{delay} of the denotation of the calculation at point $\beta$ is the least ordinal $\gamma$ such that for all other calculations $\{(f'_\delta , a'_\delta , d'_\delta)\}_{\delta < \alpha'}$ in $\Omega$, if $(f_\delta , a_\delta) = (f'_\delta , a'_\delta)$ for all $\delta < \beta + \gamma$, then $d_\delta = d'_\delta$ for all $\delta \leq \beta$. }\end{definition}
 The delay tells us for how much longer we must run a calculation before we can tell  what the denotation will be.

 \bigskip

\noindent A key property of a d-procedure is that we can use the denotations  to code the procedure  in a manageable way as a subset of the continuum. 
\begin{definition}{\em \hspace*{2mm}
\begin{itemize}
\item[a)] Let $\Omega$ be a d-procedure. The \emph{representation} of $\Omega$ will be the set of quadruples $(D,\prec , \{(f_d,a_d)\}_{d \in D},c)$ derived from calculations $(\{(f_\beta , a_\beta , d_\beta)\}_{\beta < \alpha},c)$ in $\Omega$ as follows:
\begin{itemize}
\item[i)] $D$ is the set of $d_\beta$ for $\beta < \alpha$ and $\prec$ is the corresponding ordering on $D$.
\item[ii)] When $d = d_\beta$, $f_d$ is the $f_\beta$ and $a_d$ is the $a_\beta$ of the calculation.
\item[iii)] $c$ is the value of the calculation.
\end{itemize}
We code these items as elements of $\N^\N$ in some standard way.
\item[b)] We say that a d-procedure $\Omega$ is $\Pi^1_1$ if the representation of $\Omega$ is a $\Pi^1_1$-set.
\item[c)] If $(D,\prec , \{(f_d,a_d)\}_{d \in D},c)$ is a calculation in a d-procedure and $(D',\prec' , \{(g_d,b_d)\}_{d \in D'})$ satisfies that $\prec'$ is a well ordering of $D'$, each $g_d$ is of type 1 and each $b_d$  is of type 0, we say that $(D',\prec' , \{(g_d,b_d,\}_{d \in D'})$ is \emph{isomorphic to} an initial segment of $(D,\prec , \{(f_d,a_d)\}_{d \in D})$ if there is an order isomorphism $\rho$ from $D'$ to an initial segment of $D$ such that $f_{\rho(d)} = g_d$ and $a_{\rho(d)} = b_d$ for all $d \in D'$.

\end{itemize}}\end{definition}
\begin{lemma}\label{Lemma4.2} The functional $\ci$ is definable from a  d-procedure that is $\Pi^1_1$. \end{lemma}

\begin{proof}

For each $b$, we will construct a procedure for the 0 - 1-valued function \[\lambda G. \ci(F_G)(b),\] where $F_G(f)(a) = \min\{1,G(a\what f)\}$ is as in Remark \ref{Remark.2.1}. The only difference between these procedures will be in the value part, the $c$ in each string. 
\smallskip

Let $G$ be given. We will describe the  calculation with denotation that will match $G$ and conclude with the value $\ci(F_G)(b)$.
Let  $\{f_\beta\}_{\beta \leq \alpha}$ be the sequence constructed while defining $\ci(F_G)$.

For each $\beta \leq \alpha$, let $g_{a,\beta} = a \what f_\beta$. We see that in order to ``compute" $\ci(F_G)(b)$ we have to evaluate $G$ on all functions $g_{a,\beta}$ for all $\beta \leq \alpha$, a sequence of queries of order type $\omega(\alpha + 1)$. So, we define the  calculation matching $G$ as
\[(\{( h_\gamma , b_\gamma , d_\gamma)\}_{\gamma < \omega(\alpha + 1)} , c) \] where
\begin{itemize}
\item[-] $h_{\omega \cdot \beta + a} = g_{a,\beta}$ for $\beta \leq \alpha$ and $a < \omega$.
\item[-] $b_{\omega \cdot \beta + a} = G(g_{a,\beta})$ for $\beta$ and $a$ as above.
\item[-] $d_{\omega \cdot \alpha + a} = \langle 0,a \rangle$ for $a \in \omega$.
\item[-] $d_{\omega \cdot \beta + a} = \langle x+1 , a \rangle$, where $x$ is minimal such that $f_{\beta+ 1}(x) = 1$ while $f_\beta(x) = 0$, if $\beta < \alpha$ and $a \in \omega$.
\item[-] $c = f_\alpha(b)$.
\end{itemize}
It remains to prove that the representation is $\Pi^1_1$. We do this through the following steps:
\begin{enumerate}
\item Since the set of pairs $(D,\prec)$ where $D \subseteq \N$ and $\prec$ is a well ordering of $D$ is $\Pi^1_1$, the set $\Omega_1$ of quadruples $(D,\prec,\{(h_d,b_d)\}_{d \in D},c)$ where $(D,\prec)$ is a well ordering as above is $\Pi^1_1$. We call the elements in $\Omega_1$ \emph{strings}.
\item If $\Omega_2$ is is the set of strings in $\Omega_1$ where the order type of  $(D,\prec)$ equals  $\omega \cdot(\alpha + 1)$ for some $\alpha$, we still have a $\Pi^1_1$-set.
\item Let $\Omega_3$ be the strings in $\Omega_2$ that corresponds to a possible evaluation of $\ci$ on some $G$. This requires that the calculation is \emph{locally correct}, i.e.  that each $(h_d,a_d)$ is in relation to its $(D,\prec)$-predecessors as prescribed by the recursion step. This can be decided arithmetically, so $\Omega_3$ is also $\Pi^1_1$.
\item For a string in $\Omega_3$, we can arithmetically decide if the enumeration $(D,\prec)$ is as in the construction above, so the representation $\Omega_4$ of the calculations with denotations in the procedure for $\ci$ will also be $\Pi^1_1$. \end{enumerate}
\end{proof}
\begin{remark}{\em We introduce delays in this construction. Whenever we simulate one step in the induction, we must wait until we know if we are at the final step or not before we can decide what the denotation will be, and this involves a delay of length $\omega$.}\end{remark}

\begin{definition}{\em Let $\Phi$ be a total functional of type 3. We say that $\Phi$ is \emph{$\Pi^1_1$-definable} if $\Phi $ is definable from a $\Pi^1_1$ d-procedure.}\end{definition}
\begin{lemma}\label{Lemma.4.4}
The class of $\Pi^1_1$-definable total functionals of type 3 is closed under relative Kleene computability.
\end{lemma}

\begin{proof}
We build on the proof of Theorem \ref{3.4} and the construction in Definition \ref{Def.3.3}. We just have to show how to add the denotations $d_\beta$ to each item in the calculation, and then show that the complexity of the representation is preserved. We define the d-procedure as follows:
\smallskip

 In the cases of initial computations there are no ordinals to be denoted, and in the case of S9 we can keep the denotations as they are.
 \smallskip
 
 In the case of composition, we can use $d \mapsto \langle 0,d\rangle$ to denote the items in the first part and the map $d \mapsto \langle 1,d \rangle$ to denote the items in the second part.
 \smallskip
 
 When we compute $g$ and then apply $F$ to $g$, we use the map $d \mapsto \langle c+1,d\rangle$ to  denote the items coming from the calculation computing $g(c)$ and end the full subcalculation with $(g,F(g),\langle 0,0 \rangle)$.
 \smallskip
 
 In the case where we apply the procedure $\Omega_i$ for $\Phi_i$ to a partial functional $H$ of type 2 for which we have an index, our calculation will be the concatenation of the calculatioins related to the computations of $H(g_\beta) = a_\beta$, where we also inserted $(g_\beta ,\ast)$ in front of each such local calculation.  If the $\Omega_i$-denotation for the pair $(g_\beta , a_\beta)$ in the calculation evaluating $\Phi_i(H)$ is $d_1$, we use $\langle d_1,0\rangle$ to denote $(g_\beta , \ast)$ in the calculation we construct, and if an item in the calculation defined from the computation of $H(g_\beta) = a_\beta$ is $d_2$, we let $\langle d_1,d_2+1\rangle$ be the denotation of the corresponding item in the concatenated calculation.
\smallskip

It is clear that if two calculations, as in Definition \ref{Def.3.3} are equal, the denotations will be the same as well. This defines a d-procedure.

\medskip

 It remains to show that the representation of this d-procedure will be $\Pi^1_1$ when the representations of the d-procedures for $\Phi_1 , \ldots , \Phi_n$ are $\Pi^1_1$. This will be the hard, technical part of our proof, and we first give a brief explanation of what we aim to do:

\smallskip

We let the $\Pi^1_1$-representations for $\vec \Phi$ be given. Using the recursion theorem for computing relative to $^2E$, we will design an algorithm that, given $e$, $\vec f$, $\vec a$ and a representation \[(D,\prec , \{(f_d,a_d)\}_{d \in D} )\] of a string with denotations (a d-string for short) will semi-check, in the sense of providing an algorithm relative to $^2E$ that terminates when the property holds,    if there is some $F$ such that this string matches $F$ and that the d-string gives us  the representation of the  calculation  we constructed for the computation $\{e\}(\vec \Phi,F,\vec f , \vec a)$.  In addition, if our algorithm finds the  representation $(D,\prec , \{(f_d,a_d)\}_{d \in D} )$   adequate as the representation of a d-calculation, it will produce the value of the computation $\{e\}(\vec \Phi,F,\vec f , \vec a)$, which then will be the same for any $F$ matching the given d-string (which by now is confirmed as a d-calculation). Since termination of $^2E$-algorithms is of complexity  $\Pi^1_1$, this will prove the lemma. Without stressing this point everywhere needed, we assume that the given d-string matches itself, in the sense that if both $(f,a)$ and $(f,a')$ occur, maybe at different places, then $a=a'$.

\medskip

 As usual, our $^2E$-procedure will be defined by cases following S1 - S9, where we only focus on the nontrivial cases.
 \smallskip

If $e$ is an index for an initial computation, we check if the given string is empty. If so, it is fine as a calculation, and we can read off the value from the index, the given $\vec f$ and $\vec a$.
\smallskip

\noindent Composition:\[\{e\}(\vec \Phi , F , \vec f , \vec a) = \{e_1\}(\vec \Phi , F , \vec f , \{e_2\}(\vec \Phi , F , \vec f , \vec a),\vec a).\]
First we check if $(D,\prec)$ is of the form $(D_2,\prec_2) + (D_1,\prec_1)$ where each $d \in D_2$ is of the form $\langle 0,d'\rangle$ and each $d \in D_1$ is of the form $\langle 1,d'\rangle$.

Let $D_2' = \{d' : \langle 0 , d'\rangle \in D_2\}$ and consider the corresponding string inherited from the given one. If this is ok for the computation $\{e_2\}(\vec \Phi , F , \vec f , \vec a)$, we compute the value $c'$, and now ask if the $(D_1,\prec_1)$ is ok for $\{e_1\}(\vec \Phi , F , \vec f , c', \vec a)$ in the same sense.
\smallskip

 \noindent Application of $F$:
\[\{e\}(\vec \Phi , F , \vec f , \vec a) = F(\lambda c.\{e_1\}(\vec \Phi , F , \vec f , c , \vec a)).\]
First we check if the given string has a last element $(g,a,\langle 0,0\rangle)$ and if what comes before can be seen as an $\omega$-sum of intervals $I_c$ where the denotations are of the form $\langle c+1,d\rangle$.

If this is the case, the given string generates, in analogy with the case for composition, strings $t_c$, and we check for each of them if they are ok for the computation $\{e_1\}(\vec \Phi , F , \vec f , c , \vec a)$, and with value $g(c)$. 

If they are all ok we accept the given string  as a calculation, and see that the value of $F(\lambda c.\{e_1\}(\vec \Phi , F , \vec f , c , \vec a))$ must be $a$.
\smallskip

\noindent Application of $\Phi_i$: 
\[\{e\}(\vec \Phi , F , \vec f , \vec a) = \Phi_i(\lambda g.\{e_1\}(\vec \Phi , F , g , \vec f , \vec a)).\]
This is where we need the extra information stored in the LOG. We proceed as follows:

In the given string, first check if $(D,\prec)$ is the union of intervals where the first element of the interval is of the form $\langle d_1,0\rangle$ and the rest are of the form $\langle d_1,d_2 + 1\rangle$.

Then check for each of these intervals , where $g = f_{\langle d_1 , 0\rangle}$, if the rest of this interval, after replacing $\langle d_1,d_2 + 1\rangle$ with $d_2$, is ok for $\{e_1\}(\vec \Phi , F , g , \vec f , \vec a)$, and if so, compute the value $a$.

Finally, we collect these pairs $(g,a)$ with denotation $d_1$ into a string, and check if this is a calculation in  $\Omega_i$ with some value $c$. For this, we use Gandy selection, and we then find the correct value as well.
\medskip

In order to complete the argument we must prove that if this process works, then the computation in question, relative to any $F$ matching the given string, will terminate with the chosen value, and prove that if the computation terminates for a total $F$, then our process terminates on the corresponding  representation of the d-calculation, and again, that it gives the right value. Both arguments are by induction on the length of computations, the first for $^2E$-computations and the latter for the computation of $\{e\}(\vec \Phi,F,\vec f , \vec a)$. The details are trivial.
\end{proof}
\section{Inductive procedures}\label{5.}
As a consequence of Lemmas \ref{Lemma4.2} and \ref{Lemma.4.4} we see that all total functionals of type 3 computable in $\ci$ will be definable from a d-procedure that is $\Pi^1_1$, but the converse is not true, see Theorem \ref{Theorem.5.10}.
\smallskip

\noindent
The aim of this section is to narrow down a subclass of the d-procedures further in order to approach a characterisation of the class we are primarily interested in, the functionals computable in $\ci$.
\subsection{Computability in $^2E$}
Matters are trivial if the d-procedure is hyperarithmetical:
\begin{theorem}
Let $\Phi$ be of type 3. Then $\Phi$ is computable in $^2E$ if and only if $\Phi$ is definable from a d-procedure with a $\Delta^1_1$-representation.\end{theorem}
\begin{proof}
First let $\Phi$ be definable from the d-procedure $\Omega$, and assume that the representation is $\Delta^1_1$. By the boundedness theorem for $\Sigma^1_1$-sets of codes for ordinals, see e.g.  \cite[Exercise II 5.9]{Sacks}, there will be a computable ordinal $\lambda$ such that all calculations in $\Omega$ have order-types bounded by $\lambda$. Let $(X,\lhd)$ be a computable well-ordering of length $\lambda$, and for each $x \in X$, let $X_x = \{y \in X : y \lhd x\}$. For each $F$, and by recursion on the $\lhd$-rank of $x \in X$, we will use $F$ and $^2E$ to \emph{compute} a string indexed by $X_x$ that matches $F$ and is, modulo the choice of denotations, isomorphic to the calculation in $\Omega$ matching $F$, until $\Omega $ tells us what the value $\Phi(F)$ must be. We use the recursion theorem, and explain the step from $x$ to its $\lhd$-successor $x'$. So, as an induction hypothesis, we assume that we have constructed the string $t = \{(f_y , a_y,y)\}_{y \lhd x}$. This string is \emph{isomorphic} to an initial segment of $((D,\prec,\{(f'_d,a'_d)\}_{d \in D},c)$ if there is a $d \in D$ with the same rank as $x$, and the corresponding isomorphism $r$ from $X_x$ to $D_d$ will satisfy that $f_y = f'_{r(y)}$ and $a_y = a'_{r(y)}$ for all $y \in D_x$. Now, the set $\Omega_t$ of calculations in $\Omega$ such that the string $t$ is isomorphic to an initial segment will be $\Delta^1_1$ relative to $t$. 
\smallskip

In order to know what to do next, we first have to split between the two cases: are we able to give out a value for $\Phi(F)$ or must we continue the evaluation, that is, identifying, up to isomorphism,  a larger part of the calculation matching $F$? Since $\Omega$ has a calculation for  $\Phi(F)$, we know that there is at least one calculation in $\Omega$ of which $t$ is isomorphic to an initial segment. Moreover, if $t$ is actually isomorphic to a calculation in $\Omega$, this is unique, and $\Omega$ provides us with the value. So, in order to decide between the two cases, we ask a $``\Delta^1_1(t)"$-question, i.e. a $\Sigma^1_1(t)$-question and a $\Pi^1_1(t)$-question that are equivalent. Those are: \begin{center}`is $t$ isomorphic to a proper initial segment of some element of  $\Omega_t$?'\end{center} and  \begin{center}`is $t$ isomorphic to proper initial segments of all elements in $\Omega_t$?' .\end{center}  In the case the answer is `no', we have constructed a copy of the calculation matching $F$. Then we can compute the unique value $\Phi(F) = c$ from the data. On the other hand, if the answer is `yes' ,  there will be a next query $f_x$ that will be unique for all calculations  in $\Omega_t$.  $\{f_x\}$ is a $\Delta^1_1$-singleton relative to $t$, and we can compute each $f_x(n)$ from $t$ and $^2E$. In both cases, we can rely on Gandy selection. 
This proves the theorem one way.

\medskip

 Now assume that $\Phi(F) = \{e\}(^2E,F)$. We can construct a d-procedure $\Omega$ for $\Phi$  in analogy with the one we constructed in the proof of Lemma \ref{Lemma.4.4}, without relativizing the construction to a set of $\vec \Phi$ with  $\Pi^1_1$-procedures. That the result now will be $\Delta^1_1$ can be seen from the following consideration:
\begin{itemize}
\item[i)] Whenever $\{e\}(^2E,F)\!\!\downarrow$ we can compute the corresponding  calculation
\[(D,\prec ,\{(f_d,a_d)\}_{d \in D},c)\] uniformly in $e$, $F$ and $\exists^2$, by use of the recursion theorem. It is worth noticing that there will be no delay here, given $\{f_{d},a_{d'}\}_{d' \prec d}$, we do not only have a unique value for the next $f$, but also for its denotation $d$, even if the calculations  are matching different $F$s.

\item[ii)] Next we observe that when checking if a representation $(D,\prec , \{f_d,a_d\}_{d \in D})$ of a d-string is a real representation of a real calculation of a value, we can relax the requirement that $(D,\prec)$ is a well ordering. The checking the way we did it in the proof of Lemma \ref{Lemma.4.4} is partially computable by the recursion theorem, and can be proved to terminate for a given $e$ under the assumption that there is at least one total $F$ matching the given d-string such that $\{e\}(^2E,F)\!\!\downarrow$. 
\end{itemize}
This shows that the  d-procedure will be $\Delta^1_1$ in this case.\end{proof}
\begin{remark}{\em We have essentially used that the element of a $\Delta^1_1$-singleton is itself hyperarithmetical, and this implicitly provides us with a {\sc next}-function in this case.
\newline
There is no delay in the $\Delta^1_1$-d-procedure constructed in the above argument, the denotations in an initial segment of a calculation is uniquely determined by the segment itself.}\end{remark}

\subsection{The class ${\bf W}(\ci)$}
In this section we will give a closer analysis of the class of functionals of type 3 that are computable in $\ci$. 
We will do so by investigating  additional properties of the elements in the following class:
\begin{definition}\label{Def.5.4}{\em Let ${\bf W}(\ci)$ be the set of d-procedures for functionals $\Phi$ computable in $\ci$ as constructed in the proofs of Lemma \ref{Lemma4.2} and Lemma \ref{Lemma.4.4}.
}\end{definition}
\subsubsection{Tame d-procedures}
In this sub-section we will introduce two properties shared by all d-procedures in ${\bf W}(\ci)$.
\begin{definition}{\em 
 Let $\Omega$ be a procedure defining a total functional. 
 \begin{itemize}
 \item[a)]Let  $\Omega_{\rm pre}$ be the set of of  $(D,\prec,\{(f_d,a_d)\}_{d \in D})$ that are isomorphic to an initial segment of a calculation  in $\Omega$.
 \item[b)] Let {\sc next}$_\Omega$ be the function mapping $ t \in \Omega_{\rm pre}$ to the disjoint union of $\N^\N$ and $\N$ such that \begin{itemize}
 \item[-] If $t$ is isomorphic to a calculation $t'$ in $\Omega$, then {\sc next}$_\Omega(t)$ is the value of this calculation
 \item[-] If $t$ is isomorphic to a proper initial segment $t'$ of a calculation $t''$ in $\Omega$, then {\sc next}$_\Omega(t) = (f,c)$ where $F(f) = ?$ is the next query after $t'$ in $t''$ (independent of the choice of $t''$) and $c \in \{0,1\}$. Moreover, if $c = 0$, then the  next pair in $t''$ after $t'$ will be of the form $(f,\ast)$ while if $c = 1$, we continue $t'$  with a pair $(f,a)$ in $t''$ for some $a \in \N$.
  \end{itemize} 
 \end{itemize}

}\end{definition}
\begin{definition}{\em Let $\Omega$ be a d-procedure for a total functional. We say that $\Omega$ is \emph{tame} if $\Omega$ is $\Pi^1_1$, $\Omega_{\rm pre}$ is $\Pi^1_1$ and {\sc next}$_\Omega$ is partially computable in $^2E$.}\end{definition}
\begin{lemma}\label{Lemma.5.2} Let $\Omega$ be the d-procedure for $\ci$. Then $\Omega$ is tame.\end{lemma}
\begin{proof}
We need the full complexity of $\Pi^1_1$ to formulate that we are dealing with well orderings $(D,\prec)$, the rest is actually arithmetical. Each step in the underlying recursion takes $\omega$ many steps when we evaluate according to $\Omega$. It is clearly arithmetical to decide if any ordered set of pairs $(f,a)$ indexed over $\N$ locally satisfies the recursion in $\ci$, so checking if a d-string is in $\Omega_{\rm pre}$ is arithmetical when we know that the representation is well ordered. If a string is locally correct, the next query is arithmetically defined if there is one, and the value is arithmetically expressible from the  list of queries answers if the d-string corresponds to a calculation, so {\sc next}$_\Omega$ is computable in $^2E$ as requested.
\end{proof}
\begin{lemma} 
The class of functionals definable from tame d-procedures $\Omega$  is closed under Kleene computability.
\end{lemma}
\begin{proof}
Let $\vec \Phi = \Phi_1, \ldots ,\Phi_n$ be defined from the tame d-procedures $\Omega_1 , \ldots , \Omega_n$. We already know that the d-procedure for any $\Phi$ computable in $\vec \Phi$ is $\Pi^1_1$.
\smallskip

We use the recursion theorem to define, for each index $e$ and extra inputs $\vec f$ and $\vec a$, a set $X_{e,\vec f , \vec a}$ of d-strings \[((D, \prec , \{(f_d,a_d)\}_{d \in D})\] 
that is  $\Pi^1_1$ uniformly in $e,\vec f , \vec a$, together with the function
\begin{center}{\sc next}$_{e, \vec f , \vec a}$\end{center}
defined on $X_{e,\vec f , \vec a}$, and show that 
\begin{itemize}
\item[i)] If $F$ is of type 2 and $\{e\}(\Phi_1 , \ldots , \Phi_n, F , \vec f , \vec a) = c$ and $(D,\prec , \{(f_d,a_d)\}_{d \in D} , c)$ is the associated d-calculation obtained from the proof of Lemma \ref{Lemma.4.4}, then any string isomorphic to an initial segment of $(D,\prec , \{(f_d,a_d)\}_{d \in D} )$ is in $X_{e,\vec f , \vec a}$.
\item[ii)] If a string is in $X_{e,\vec f , \vec a}$ and matches some $F$ for which $\{e\}(\Phi_1 , \ldots , \Phi_n, F , \vec f , \vec a)$ terminates, then the string is isomorphic to an initial segment  of the d-calculation  obtained through the proof of Lemma \ref{Lemma.4.4}.
\item[iii)] {\sc next}$_{e,\vec f , \vec a}$ is computable in $^2E$ uniformly in the parameters and acts as specified.
\end{itemize}

\smallskip

We define $X_{e,\vec f , \vec a}$ and {\sc next}$_{e , \vec f , \vec a}$ by cases according to the scheme $e$ represents. Note that since we are dealing with semi-decidable sets, we cannot take NO for an answer, and search-procedures have to use Gandy selection:

\smallskip

- $e$ is an index for an initial computation given by S1-S3, S7: $X_{e,\vec f , \vec a}$ will consist of the empty string only. {\sc next}$_{e, \vec f , \vec a}$ is trivially given in all these cases.

\smallskip

- Composition: \[\{e\}(\vec \Phi , F , \vec f , \vec a) = \{e_1\}(\vec \Phi , F , \vec f , \{e_2\}(\vec \Phi , F , \vec f , \vec a ) , \vec a).\]
Given $(D, \prec , \{(f_d , a_d)\}_{d \in D})$, where $(D,\prec)$ is a well ordering, we use $(D,\prec)$ recursion to test if the initial segments of $(D, \prec , \{(f_d , a_d)\}_{d \in D})$ are in $X_{e_2 , \vec f , \vec a}$ until we either found an initial segment  that is in  $X_{e_2 , \vec f , \vec a}$ and with a value $c$ or we find that the given string is in  $X_{e_2 , \vec f , \vec a}$, and thus in $X_{e,\vec f , \vec a}$. If this search fails, the given string is not in  $X_{e , \vec f , \vec a}$, and non-termination is not a problem. If this search ends with a proper substring that is in  $X_{e_2 , \vec f , \vec a}$ and with a value $c$, we compare the rest of the string with  $X_{e_1 , \vec f , c, \vec a}$ in the same way. The {\sc next}-function for $e$ will be inherited from the {\sc next}-functions of $e_2$ and $e_1 , c$, and these can be used to check that the given string does not go too far, beyond where we should have a valued string.
\smallskip

- We leave the cases for permutation and S9 for the reader, as those cases are even simpler.
\smallskip  

-  Application of $F$:
\[\{e\}(\vec \Phi , F , \vec f , \vec a) = F(\lambda a. \{e_1\}(\vec \Phi , F , \vec f , a , \vec a)).\]
In analogy with how we dealt with composition, we can compare a given string poin-by-point with elements in $X_{e_1 , \vec f , 0 , \vec a}$ , in $X_{e_1 , \vec f , 1 , \vec a}$ and so forth until we either find that the given string is a concatenation of finitely many strings in these sets, all except the last one maximal, that it is the concatenation of one maximal string from each $X_{e_1 , \vec f , a , \vec a}$ in increasing order or that it even contains a final $(g,b)$ at the end. In the last case, we also must check if $g(a)$ is the value of the string from $X_{e_1 , \vec f , a , \vec a}$ used in the concatenation before accepting the given string. When accepted, we inherit the {\sc next}-function in the obvious way.
\smallskip 

- Application of $\Phi_i$:
\[\{e\}(\vec \Phi , F , \vec f , \vec a) = \Phi_i(\lambda g.\{e_1\}(\vec \Phi , F , g , \vec f , \vec a)).\]
Here it may be useful to look back on Definition \ref{Def.3.3}. We explain informally how we, point by point, compare the initial segments of the given string $(D,\prec , \{(f_d,a_d)\}_{d \in D})$ with the requirements for $\Omega_i$ and the various sets \[X_{e_1,g,\vec f , \vec a}\] where we may assume that we have defined these sets as a part of  the induction hypothesis:
\begin{itemize}
\item If $d_0$ is the $(D , \prec)$-least element, $f_{d_0}$ has to be the first query $g_0$ in $\Omega_i$, given to us by {\sc next}$_{\Omega_i}$, with $a_{d_0} = \ast$.
\item We check the next segment of $(D,\prec , \{(f_d,a_d)\}_{d \in D})$ for membership in $X_{e_1,g,\vec f , \vec a}$ until we have reached a value or until the given string is exhausted. 
\item In the latter case, the string is in $X_{e,\vec f , \vec a}$ and in the first case, we let $b_0$ be the value, check that the  the pair $(g_0 , b_0)$ is in in ${\Omega_i}_{\rm pre}$ and use the {\sc next}-function of $\Omega_i$ to verify that the next query in $\Omega_i$ will be the next query in $(D,\prec , \{(f_d,a_d)\}_{d \in D})$. 
\item By transfinite recursion on $(D,\prec)$ we can continue this comparison until the given string is exhausted or until it does not compare with strings in $\Omega_i$ (whenever we have found a new value there, and can use its {\sc next}-function to find the new $g_\beta$), or with the strings in the sets  $X_{e_1,g_\beta,\vec f , \vec a}$.

\end{itemize}
It is now routine to verify the  properties i) - iii). \end{proof}

\subsubsection{Blocking}
It is not the case  that all functionals definable by a tame d-procedures will be computable in $\ci$, see Theorem \ref{Theorem.5.10}. The point with the denotations is that they may make it easier to design non-monotone inductions that are copying evaluations in a procedure, but the obstacle will be that we will not, in general, be able to tell from a part of a calculation what the correct denotation of the next query will be, there may be a  \emph{delay} as defined in \ref{def.delay}. We find this phenomenon in the procedure for  $\ci$, where we first must establish, within each in  a series of $\omega$-length sub-calculations, if we reached the final fixed point or not before reading off the correct denotation. That the situation would be simpler without this  obstacle, is seen from the following lemma:
\begin{lemma} \label{obs.2} Let $\Omega$ be a d-procedure for a total functional such that for any  sub-string $\{(f_\beta , a_\beta , d_\beta)\}_{\beta < \alpha}$ of a calculation in $\Omega$, the denotations $d_\beta$ are uniquely given by $\{(f_\beta , a_\beta)\}_{\beta < \alpha}$.
\newline
If $\Omega$ in addition is tame, and the unique choice of $d_\beta$ is computable from $\{(f_\gamma,a_\gamma\}_{\gamma \leq \beta}$ and $^2E$ uniformly at each stage, then the functional defined by $\Omega$ is computable in $\ci$.\end{lemma}
\begin{proof}
We code an entry $(f_\beta , a_\beta , d_\beta)$ as the set \[\{\langle \bar f_\beta(n) , a_\beta , d_\beta\rangle : n \in \N\}.\]
These sets will be disjoint, so we may code each initial segment of a string as a pair of sets, where one is the union $X_\beta$ of such single codes and the other is the corresponding ordering of the denotations. We will use that the Suslin functional $\su$ is computable in $\ci$.
\newline
Given $F$, we design a non-monotone inductive definition $\Gamma_F$ computable in $\su$ that simulates the evaluation of the calculation matching $F$:
\begin{enumerate}
\item Given $X$, we use $\su$ to check if $X$ codes an initial segment of a calculation in $\Omega$ as coded above. If not, let $\Gamma_F(X) = X$, and if it does, continue.
\item Use {\sc next}$_\Omega$ to find the next query $f_\alpha$, use $F$ to find $a_\alpha = F(f_\alpha)$ and finally the $^2E$ algorithm that gives us the unique denotation $d_\alpha$.
\item Let $\Gamma_F(X)$ be $X$ extended with the code for $(f_\alpha , a_\alpha , d_\alpha)$ in the set to the left and all pairs $\langle d_\beta , d_\alpha\rangle$ for $\beta < \alpha$ in the set to the right.
\end{enumerate}
It is clear that the set $\ci(\Gamma_F)$ will code the calculation in $\Omega$ matching $F$, together with the ordering of all the denotations used in that calculation, and we can use {\sc next}$_\Omega$ to compute $\Omega(F)$.
\newline
Further details are left for the reader. \end{proof}
\bigskip

  We will now add  further structure to d-procedures, \emph{blocks}. It will be like  inserting commands of the form $\backslash\!$begin\{{\tt block}\} and $\backslash$\!end\{{\tt block}\} bracketing blocks and sub-blocks. These imaginary commands must satisfy, for each calculation, the standard rules of bracketing, allowing for infinite branchings in the length, but only finite nesting in depth. Where to put these commands will determined by the initial segment of the string up to where the command is, and the use will be to mark that there is now an uncertainty to what the denotations will be at the end, and that we have to carry out a sub-procedure, or evaluate the calculation for $F$ a bit further,  in order to find the true denotations of the calculation. We will consider two examples before giving the abstract definition of a tame d-procedure with blockings:
\begin{example}{\em Let $\Omega$ be the d-procedure for $\ci$. Recall that, given $G$, $\Omega(G) $ will iterate the induction given by $F_G$, generating the sequence $\{f_\beta\}_{\beta \leq \alpha}$ by evaluating $G$ on $0 \what f_0,1 \what f_0 , \ldots_\omega 0 \what f_1 , 1 \what f_1 \ldots_{\omega(\alpha + 1)}$.
\newline
Each $\omega$-sequence will be a block in this case, and after each block we know what the denotation for the query $G(n \what f_\beta) = \;?$ will be, but not while we are inside a block. However, in order to view the calculation within a block as a sub-procedure, we only need denotations that are unique for queries within this block, and ignore the larger picture. It is not hard to modify the proof of Lemma \ref{obs.2} to see that we can simulate the procedure for $\ci$ using $\ci$. 

}\end{example}
\begin{example}{\em In the case of composition we constructed the calculations as concatenations of strings for the two parts, and when defining the new denotations, we paired the denotations from the first part with 0 and from the second part with 1. We may consider the first part as one block and the other part as another one, but if we from the larger picture know that we are entering a composition, there is no need for this. There is no delay in deciding what the  denotations are inherited in the construction of denotations for compositions, as there is for the construction of denotations for transfinite recursions with an unknown end.

}\end{example}
We will need the blocking structure on calculations to characterise functionals of type 3 computable in $\ci$ in terms of procedures. 

The blocks will be organised in a nested way, with some blocks being sub-blocks of others. The point is that, within each block, we will define denotations along the way, and when the need of a delay is observed, we enter a sub-block where we form temporary denotations that at the end of the block will be rewritten to the true ones. The nesting of the blocks will reflect that there may be  delays within a period of delay, so the rewriting of denotations may go through several levels.

We will now give the full definition:
\begin{definition}{\em 
Let $\Omega$ be a tame d-procedure.
\begin{itemize}
\item[a)] A \emph{block} in $\Omega$ is an interval $t_2$ in a calculation $t_1 \what t_2 \what t_3$ in $\Omega$. Blocks $t_2$ in $t_1 \what t_2 \what t_3$ and $t'_2$ in $t'_1 \what t'_2 \what t'_3$ are considered to be equal if $t_1 \what t_2 = t'_1 \what t'_2$.

\item[b)] A \emph{blocking} of $\Omega$ is a set of blocks for each calculation $t$ in $\Omega$ satisfying:
\begin{itemize}
\item[i)] Given two blocks in $t$, they are either disjoint or one is included in the other.
\item[ii)] For each calculation, all chains of blocks totally ordered by inclusion will be finite. 

\item[iii)] $t$ is one of the blocks in $t$.
\item[iv)] The \emph{level} of a block $t_1$ in $t$ is the number of other blocks in $t$ properly containing $t_1$ as a substring.
\item[v)] If $t = t_1\what t_2$ and $t' = t_1 \what t'_2$ are two calculations in $\Omega$ with a common initial substring $t_1$, and if a block of level $m$ in $t$ starts at the beginning of $t_2$, then a block of level $m$ starts in $t'$ at the beginning of $t_2'$. 

Moreover, if the two blocks coincide until one of them ends, they are equal.

\end{itemize}
iii) above just makes the rest easier to express.
\item[c)] The blocking is \emph{tame} if we in addition have 
\begin{itemize}
\item[i)] There is a partial function {\sc block}$_\Omega$ computable in $^2E$ such that  for each string $t$ in $\Omega_{\rm pre}$, {\sc block}$_\Omega(t)$ decides for each $m$ if there is a  block of level $m$  starting at the next query {\sc next}$_\Omega(t)$ and decides the levels of the blocks, if any,  ending before the next query. 
\item[ii)] For each block $s = \{(f_\delta , a_\delta)\}_{\gamma \leq \delta < \beta}$ in a calculation $t$ there is a unique injective denotation $\{d^s_\delta\}_{\gamma \leq \delta < \alpha}$, where these denotations are computed using the two functions {\sc denote} and {\sc redenote} (with subscript $\Omega$ if needed) both computable in $^2E$ and where
\begin{itemize}
\item if $t_1 \what (f_\beta , a_\beta)$ is an initial substring of the calculation $t$ and $s$ is the block of highest level containing $(f_\beta , a_\beta)$ then {\sc denote}$(t_1 , a_\beta)$ will be the denotation $d^s_\beta$.

\item If $s$ is a block of level $m > 1$ contained in the calculation $t = t_1 \what s_1 \what s \what s_2\what t_3$, where  $t_2= s_1\what s \what s_2$ is the block of level $m-1$, then {\sc redenote} with input $t_1 \what s_1 \what s$ and $m$  will give us $d^{t_2}$ restricted to $s_1 \what s$.

\end{itemize}
\end{itemize}
\end{itemize}

}\end{definition}
\begin{comment}{\em It is c), ii) that captures the essence of blocking, a block represents the local delay of deciding what the denotation one level up will be like, and will make it possible to simulate the evaluation of a procedure as a nested application of $\ci$ .

}\end{comment}
\begin{definition}{\em An \emph{Inductive Procedure} is a tame $\Pi^1_1$-procedure with a tame blocking.}\end{definition}
\begin{theorem}\label{Theorem.5.9} A functional $\Phi$ of type 3 is definable from an inductive procedure if and only if $\Phi$ is computable in $\ci$.
\end{theorem}
\begin{proof}
We  show first that if $\Phi$ is definable from an inductive procedure, then $\Phi$ is computable from $\ci$. We use a nested version of the argument in Lemma \ref{obs.2}, using everywhere that the functions {\sc next}, {\sc block}, {\sc denote} and {\sc redenote}  are computable in $^2E$, and thus in $\ci$.
We use the recursion theorem for Kleene computations to make the following precise:
\smallskip

Given $F$, we construct the inductive definition $\Gamma_F$ as in Observation \ref{obs.2}, using the  denotation as it comes, until we hit the beginning of a block $s$. Then we start the execution of a sub-procedure simulating the evaluation of $F$ along $s$ as an inductive definition $\Gamma_F^s$ in the same way. This sub-procedure will come to an end when the full evaluation along $s$ is simulated. At this stage we can describe the next step for $\Gamma_F$: From the output of $\Gamma_F^s$ , $^2E$ and the assumption on blocks, we can compute the correct denotations along the string up to the end of $s$, and $\Gamma_F$ just ads the full simulation of the evaluation of $F$ in $\Omega$ using this denotation.
\newline
If $s$ has sub-blocks, then $\Gamma_F^s$ will have sub-procedures in a similar way, this is why we need the recursion theorem to formally describe this procedure.

\medskip

\noindent In order to prove the other direction we elaborate on the proofs that the class of functionals satisfying our requirements is closed under Kleene computability, and assume that $\Phi_1 , \ldots , \Phi_n$ now are defined from inductive procedures.
\newline
The only case we need to consider is that of application of $\Phi_i$,

\[\{e\}(\vec \Phi , F , \vec f , \vec a) = \Phi_i(\lambda g.\{e_1\}(\vec \Phi , F , g , \vec f , \vec a)).\]
Let $G(g) = \{e_1\}(\vec \Phi , F , g , \vec f , \vec a)$. Then there is a calculation $\{(g_\beta , b_\beta)\}_{\beta < \alpha}$ in  $\Omega_i$ matching $G$. For each $\beta < \alpha$, the pair $(g_\beta,b_\beta)$ will be replaced by a substring of the composed calculation as follows: $(g_\beta , \ast)$ will just be preserved as it is, while $(g_\beta , b_\beta)$ is replaced by a string starting with $(g_\beta , \ast)$ and continued with the calculation of $F(g_\beta)$. 
\smallskip

When we defined the denotations for these composed calculations, we gave them directly from the denotations in the $\Omega_i$-calculation and from the denotations in the calculations of $G(g_\beta)$ without adding any further delay. Thus we can inherit the blocking structure of the $\Omega_i$-calculation, and whenever $(g_\beta,a_\beta)$ is in one of these blocks before we compose all the substrings, we let all blocks in $t$, where  $(g_\beta,\ast) \what t$ is inserted for $(g_\beta , a_\beta)$ and $t$ is the calculation of $G(g_\beta)$, be new blocks of a higher level.

\medskip

When we constructed the d-procedure in this case, we gave the rules for transforming the denotations in this simulating block to denotations of the corresponding items in the full calculation, and this clearly can be relativised to the blocks in $\Omega_i$. 
\smallskip

In order to tie the whole thing up showing that the definability and computability requirements of what we construct are satisfied, we need to use the recursion theorem for $^2E$, induction on the ordinal lengths of $\vec \Phi,F$-computations and induction/recursion on the level of blocks in a string. The details are tedious, but simple. \end{proof}
\subsection{The $\ci$-computable functions revisited}\label{computable}
In this section we will consider pure computations \[\{e\}(\ci , \vec a),\] without function and functional parameters.
\smallskip

In our definition of a procedure, we used the parameter $F$ to give values to queries, but at certain points we also inserted  elements of the LOG, functions appearing as arguments in sub-computations but not necessarily as arguments in the main computable function or functional we design the procedure for.
\smallskip

When transforming a computation in $\ci$ without functional arguments to a procedure this LOG will now be the backbone of the calculations. Since there will be no genuine queries anymore, we can even drop the $\ast$ for marking element-hood in the LOG. Thus a \emph{pure string} will be a triple $(D, \prec , \{g_d\}_{d \in D})$ where $(D,\prec)$ is a well ordering and each $g_d \in \N^\N$. We will consider such strings that are $\Pi^1_1$-singletons, where  the set of other strings isomorphic to initial segments of the given one is $\Pi^1_1$, where we have a {\sc next}-function computable in $^2E$ and where we have functions {\sc block}, {\sc denote} and {\sc redenote} as before, making this one-point set an inductive procedure. We call this a \emph{pure inductive  calculation}, and these pure inductive calculations will reflect  $\ci$-computations with integer inputs.

\begin{remark}{\em When transforming a computation $\{e\}(\ci, \vec a) = b$ to a pure inductive calculation we first of all linearised the computation. Then we hid some of the indexing in the function {\sc next}, and also in how we designed the denotations, we actually ``hid" all intermediate ``Kleene-calculations" that do not involve the scheme S8. However, for transfinite computations, this \emph{hiding} will not significantly reduce the length of a computation. On the other hand, when we translate a pure inductive calculation to a Kleene-computation, we may add to the length of the computation, partly because it takes time to compute {\sc next}, {\sc block}, {\sc denote} and {\sc redenote} whenever needed, and partly because we have to add time to the time-span of an induction in order to verify that the induction comes to a halt when it does.

}\end{remark}

Recall the definition of $\pi$ as the first ordinal with no code computable in $\ci$. We have the following characterisation.
\begin{lemma} $\pi$ is the supremum $\pi^*$ of the ordinal lengths  of the pure inductive calculations. \end{lemma}
\begin{proof} First, we will prove that $\pi \leq \pi^*$. Let $\alpha < \pi$, and let $(X , \prec)$ be $\ci$-computable and a well ordering of ordinal length $\alpha$. We prove this direction by constructing a pure inductive calculation with at least $\alpha$ many steps. We will use the pure inductive calculations deciding $x \in X$ and $x \prec y$ as building blocks, and with the help of those  we simulate, in the form of a grand pure inductive calculation, the induction building up $X$ one point at each step, a process that needs exactly $\alpha$ many steps.
\smallskip

Then we prove that $\pi^* \leq \pi$. Let $t$ be a pure inductive calculation. Using the same strategy we used when showing that a functional definable from an inductive procedure is computable in $\ci$, a strategy involving the recursion theorem for $\ci$, we can show that there is a nested computation relative to $\ci$ that computes a code for the ordinal length of $t$, whenever we enter a block, we compute the length of that block as a subcomputation, and then at the end of the block, ads a copy of this code to the well ordering we are building up. Actually, it will be the denotations with their ordering we compute, and in a block, the local denotations within that block.  \end{proof}
\begin{theorem}\label{lemma.car.2} Let $\alpha < \pi$. If $\alpha$ is the rank of a pure inductive calculation $t $, then $\alpha$ is $\Pi^1_1$-characterisable. \end{theorem}
\begin{proof} If $f \in \WO$ has rank $\beta$, we can decide if $\beta = \alpha$ as follows: By recursion on the wellordering coded by $f$ we can use {\sc next}$_t$ and the $^2E$-algorithm organising the blockings to compute the corresponding stages in $t$ with full information about where in the blocking structure we are at each step. If this simulation terminates exactly at the end of $t$, we accept $f$, otherwise we refute it. The set of $f$s accepted will be $\Pi^1_1$.
\end{proof}

\section{Tracing computations in $\ci$}\label{chapt.7}
\subsection{Partial procedures}
In order to make constructions of procedures smoother, we have not insisted on the natural requirement that for each $f$ and calculation, there is at most one occurrence of the query $F(f) = ?$. However, when it does appear several times, the calculation will be based on the same answer everywhere. When we refer to a query $F(f) = ?$ we will always mean the first occurrence.
When $\{e\}(\ci,F,\vec f , \vec a)\!\!\downarrow$ for all $F$, it is clear that the associated procedure will lead to terminating calculations for all $F$.  This means that when $f$ appears in a query  $F(f) = \; ?$  there will be an extension into a calculation   for all $a \in \N$. Conversely, we can prove that for every  calculation $(\{(f_\beta,a_\beta)\}_{\beta < \alpha},c)$ constructed in the procedure for $\lambda F.\{e\}(\ci , F , \vec f , \vec a)$, if $F$ matches this string, then $c$ will be the value of this computation.
\smallskip

We will now discuss what happens if we consider computations $\{e_0\}(\ci,F)$  that do not terminate for all inputs $F$. In this case, we can still define a set $\Omega_{\rm pre}$ of  strings with denotations, in these strings we may enter blocks and sub-blocks, and they behave as required for inductive procedures, since we have established these properties for each $e$, $F$ , $\vec f$ and $\vec a$ such that  $\{e\}(\ci , F , \vec f , \vec a)\!\!\downarrow$. 
\smallskip

We need to consider parameters $\vec f$ and $\vec a$, since such parameters occur in subcomputations, so we reason within this generality.
\smallskip

If we consider the construction of the inductive procedure  more carefully, we can observe what we construct in the case of non-termination more closely, again by cases according to what the index $e$ is like:
\smallskip

If $e$ is an index for an initial computation, we constructed a trivial procedure yielding the correct output without any queries made.
\smallskip

In the case of composition, we first constructed the procedure for the inner component, and for the calculations in this procedure (the strings that give us an answer), we concatenated with calculations from the procedure of the corresponding outer component. In the case the composed computation does not terminate, we do construct a string modelling  the leftmost non-terminating subcomputation.
\smallskip

 In the cases where there will be exactly one subcomputation, what we do is using the string of that one.
 \smallskip
 
 In the cases of application of $F$ or application of $\ci$, we are doing exactly as above, in case of non-termination we build up the string until we reach the leftmost non-terminating subcomputation, and ending the string in $\Omega$ with a copy of a string for this leftmost one.
 \smallskip
 
 In the case that $e$ is not an index at all, the procedure will be trivial, but with non-termination as the conclusion.
\bigskip

If $\{e\}(\ci , F , \vec f , \vec a)\!\!\uparrow$, there will be a \emph{leftmost} Moschovakis witness, a descending sequence of unsettled computations such that every computation to the left will terminate, and it will be exactly the string corresponding to evaluate $\{e\}(\ci , F , \vec f , \vec a)$ along this descending sequence of subcomputations that will be constructed in this case. If $\Omega$ is constructed like this, we will simply have some strings where the conclusion must be $\bot$ instead a proper value. However, since being a Moschovakis witness is semi-decidable, this was the original point with them, we see that $\Omega_{\rm pre}$ will still be $\Pi^1_1$. We will also have functions {\sc next}, {\sc block}, {\sc denote} and {\sc redenote} computable in $^2E$. What may be lacking is that blocks may be entered without ever being left, that we may have an infinite descending sequence of  blocks (that will then not have end points) and that we will not have a $^2E$-computable way to define the denotations for the blocks unless they have an end. So, the procedure will not be an inductive procedure. This is as it has to be, since we can define the characteristic function of a set of functions of type 2 that is complete semi-computable in $\ci$ using a procedure like this, replace the value $\bot$ with 0 and the value $a \in \N$ with 1 as the values of calculations.
\newline
These considerations contain the proof of
\begin{theorem}\label{Theorem.5.10}There is a total procedure $\Omega$ such that $\Omega_{\rm pre}$ is $\Pi^1_1$, and such that there is a partial function {\sc next}$_\Omega$ that is computable in $^2E$, but where the functional defined by $\Omega$ is not computable in $\ci$.
\end{theorem}
\begin{proof} The only property left is that we must be able to decide, using a $^2E$-algorithm, if a string $t$ in $\Omega_{\rm pre}$ is maximal or not, and in case it is maximal, if it has a value or not. By recursion on the indexing of $t$ we can follow the points in the computation tree corresponding to the points in $t$.  If this point in the computation tree has an index that is none of the indices of S1 - S9, we can conclude that there is no value. If the blocking depth is infinite, we can conclude that the string represents non-termination. This can be checked by $^2E$, using {\sc block} and calculating the $\limsup$ of the block level of the items of the string. In all other cases, there will be a next query or there will be a value given to us by the original {\sc next}-function  \end{proof}

We also have
\begin{theorem} There is a non-terminating computation $\{e\}(\ci , \vec a)$ such that the length of the string simulating the leftmost Moschovakis witness will have length at least $\pi (= \omega_1^{\ci})$. \end{theorem}
\begin{proof}
If this were not the case we can use Gandy selection for $\ci$-computations and make $\ci$-semidecidable equivalent to $\ci$-decidable, obtaining the standard contradiction by diagonalisation. 
\end{proof}

\begin{remark}{\em  Moschovakis witnesses were introduced in \cite{yannis}, and they were significant  for the understanding of higher order computing relative to normal functionals and in set recursion. They are also introduced in \cite[Section 5.2.2]{LN}, and they were applied in the proof of \cite[Theorem 6.6]{P6}.}\end{remark}
\subsection{Functionals computable in  $\ci$}\label{7.2}
In a series of papers \cite{P1,P2,P3,P4,P5,P6,P7,P8}, written jointly with Sam Sanders , we investigate classes of functionals of type 3 that serve as realisers for classical theorems in analysis. There are unsettled question concerning the relative computability of the elements of these classes. In this section we will see that in the case that elements of these classes are computable in $\ci$, we can ``almost" compute the Suslin functional from them, and consequently, they will ``almost" be computationally equivalent to $\ci$ itself. We will make the \emph{ ``almost"} precise by replacing a functional $\Phi$ computable in $\ci$ with one that traces the history of the computation, not just gives the value.  We call this the \emph{honest version of $\Phi$}
\smallskip

There is an analogue with what we do in complexity theory where the complexity of a set is often measured by the resources required to decide membership in the set and not by what we can decide using small resources combined with the set as an oracle.   In a mathematically precise way we will see that if we compute realisers for some classical theorems of analysis from $\ci$, we need the full power of $\ci$ in doing so.
\begin{definition}{\em
Let $\Omega \in {\bf W}(\ci)$  be the inductive procedure constructed from $\Phi$ as computable in $\ci$.
We define the \emph{honest version} ${\mathcal H}(\Phi)$ as the functional (of mixed type) that sends a functional $F$ of type 2 to the fixed point of the inductive definition $\Gamma_F$ as constructed from $\Omega$ in the proof of Theorem \ref{Theorem.5.9}, i.e. as the \emph{history} of the evaluation of $\Phi(F)$ from $\ci$.

}\end{definition}

An \emph{open covering} of a set $X$ in a topological space ${\mathcal T}$ is normally defined as a set $\mathcal C$ of open sets in $\mathcal T$ whose union is a superset of $X$. However, if we make use of the concept of realisers, a \emph{realiser } of the open covering will be  a map sending $x \in X$ to an open set $O_x \in {\mathcal T}$ such that $x \in O_x$. When Borel \cite{B95} gave an attempt of a direct proof of the Heine-Borel theorem, he actually, without knowing the concept,  used this idea of a realiser; with free translation he expressed his assumption as follows:
\begin{itemize}\item[(*)] Assume that we have a way of attaching an open neighbourhood $O_x$ of $x$ to each $x \in [a,b]$.\end{itemize}
In the papers with Sanders, we have considered coverings and related concepts primarily over the \emph{Cantor space} $\Ca = \{0,1\}^\N$ and the \emph{Baire space} $\Ba = \N^\N$ given as functionals $F$ of type 2, where $F(f)$ \emph{defines} the neighbourhoods $\Ca_{\overline{f}(F(f))}$ and $\Ba_{\overline{f}(F(f))}$ of extensions of $\overline{f}(F(f))$,  depending on which space we consider $f$ to be an element of.

\medskip

\noindent We have been looking at  the following three classes:
\begin{definition}{\em \hspace*{10mm}
\begin{itemize}
\item[a)] A \emph{strong realiser }for the Heine-Borel theorem will be a functional $\Theta$ such that whenever $F:\Ca \rightarrow \N$, then $\Theta(F) = \{f_1 , \ldots , f_n\}$ such that \[\Ca \subseteq \Ca_{\overline{f_1}(F(f_1))} \cup \cdots \cup \Ca_{\overline{f_n}(F(f_n))}.\]
\item[b)] A \emph{weak realiser} for the Heine-Borel theorem will be a functional $\theta$ such that whenever $F:\Ca \rightarrow \N$ then $\theta(F) = \{s_1 , \ldots , s_n\}$ where each $s_i$ is a finite binary sequence, $\{\Ca_{s_1}, \ldots , \Ca_{s_n}\}$ is a covering of $\Ca$ and for each $i = 1 , \ldots , n$ there is an $f_i \in C$ such that $\overline{f_i}(F(f_i)) = s_i$.
\item[c)] A \emph{Pincherle realiser} will be a functional $M$ such that whenever $F:\Ca \rightarrow \N$, then $M(F) = N \in \N$ and $N$ satisfies:
\begin{itemize}
\item[(-)] If $G:\Ca \rightarrow \N$ satisfies that $G(g) \leq F(f)$ whenever $\overline{g}(F(f)) = \overline{f}(F(f))$ ($F$ is considered as a realiser for local boundedness) then $G$ is bounded by $N$ on $\Ca$.\end{itemize}
\end{itemize}}\end{definition}
The following lemma is trivial:
\begin{lemma} Every strong realiser for the Heine-Borel theorem computes a weak one, and every weak realiser for the Heine-Borel theorem computes a Pincherle realiser.\end{lemma}
\noindent The proof is left for the reader.
\begin{lemma}\label{Lemma.5.14}
Let $M$ be a Pincherle realiser that is countably based. Let $F:\Ca \rightarrow \N$ be arbitrary, and let $X \subset \Ca$ be countable such that $M(G) = M(F)$ for all $G$ such that $F$ and $G$ are equal when restricted to $X$. Then $\{\Ca_{\overline{f}(F(f))} : f \in X\}$ is a covering of $\Ca$. \end{lemma}
\begin{proof}
Assume not, let $M(F) = N$ and define
\[F_1(f) = \left\{\begin{array}{ccc}F(g)&{\rm if}& g \in \bigcup_{f \in X}C_{\overline{f}(F(f))} \\N + 1 & {\rm if} & 
g \not \in \bigcup_{f \in X}C_{\overline{f}(F(f))}\end{array}\right.\]
Then $M(F) = M(F_1)$ because the two functions agree on $X$. However, if we define 
\[G(g) = \left\{\begin{array}{ccc}0&{\rm if}& g \in \bigcup_{f \in X}C_{\overline{f}(F(f))} \\N + 1 & {\rm if} & 
g \not \in \bigcup_{f \in X}C_{\overline{f}(F(f))}\end{array}\right.\] then $G$ obviously satisfies the boundedness condition induced by $F_1$, but is not bounded by $N$, contradicting that $N = M(F_1)$.\end{proof}
\begin{theorem} Let $M $ be a Pincherle realiser that is computable in $\ci$. Then the Suslin functional $\su$ is computable in ${\mathcal H}(M)$ and $^2E$. \end{theorem}
\begin{proof}
In \cite[Theorem 5.1]{P1} it is proved that there is a functional $F$ computable in $^2E$ such that the collection of neighbourhoods defined from $F$ and the hyperarithmetical binary sequences is not a covering of $\Ca$. The construction easily relativizes to an arbitrary $f \in \N^\N$ so it suffices to show how we can compute a complete $\Pi^1_1$-set from ${\mathcal H}(M)$, $F$ and $^2E$.
\newline
For a general procedure $\Omega$ and an arbitrary $G$, the calculation of $\Omega(G)$ will form a countable basis for $\Omega(G)$. If $F$ and $M$ are as given, Lemma \ref{Lemma.5.14} then shows that the calculation of  $\Omega_M(F)$ must contain queries that are not hyperarithmetical. However, in an inductive procedure, if the input functional is computable in $^2E$, then all queries appearing at the level of a computable ordinal must also be computable in $^2E$. This follows from the assumption that the {\sc next}-function is computable in $^2E$. So, the calculation of $\Omega_M(F)$ must have a rank that goes beyond $\omega_1^{CK}$. The set of (indices for the)  computable well-orderings will then be both $\Sigma^1_1$ and $\Pi^1_1$ in this calculation, and thus computable from this calculation using $^2E$. The calculation itself is computable from $F$ and ${\mathcal H}(M)$, so we are through. \end{proof}

 In \cite{N18} it is proved that $\ci$ (under the name of IND) is computable in the Suslin functional $\su$ and any strong realiser for the Heine-Borel theorem. We can improve this as
 \begin{lemma}\label{lemma.pin} Let $M$ be a Pincherle realiser. Then $\ci$ is computable in $M$ and $\su$.\end{lemma}
 \begin{proof}Let $F:\Ca \rightarrow \Ca$ be given, and consider $F$ as an inductive definition, defining the sequence $\{f_\beta\}_{\beta < \alpha}$. This set is coded as a \emph{prewellordering} $(A,\preceq)$ where $x \preceq y$ if $f_\alpha(x) = f_\alpha(y) = 1$ and  we for all $\beta \leq \alpha$ have that $f_\beta(y) = 1 \rightarrow f_\beta(x) = 1$. Identify $\preceq$ with $ \{\langle x,y\rangle : x \preceq y\}$. We will see how to compute $\preceq$ from $F$, $\su$ and $M$. We let $x,y,z,w,n,m$ etc. range over $\N$.
 
 The idea is, for each $n,x,y$ to construct a functional $G_{n,x,y}$ such that if $x \preceq y$ then $M(G_{n,x,y}) \geq n$ and such that $G_{n,x,y}$ is independent of $n$ otherwise. We will then have that \[x \preceq y \leftrightarrow \forall m \exists n (M(G_{n,x,y}) > m).\]
 We will now define $G_{n,x,y}(X)$ in cases, where we in all cases except in the last one have defined $G_{n,x,y}(X)$ independently of $n$, $x$ and $y$, and so large that $\preceq$ will be different from $X$ for at least one argument $z < G_{n,x,y}(X)$. We rename $X$ to  $\preceq_X = \{\langle z,w\rangle : \langle z,w\rangle \in X\} $. Let $\langle z,w \rangle \in X_\prec$ if $\langle z,w\rangle \in \preceq_X$ and $\langle w,z\rangle \not \in \preceq_X$.
 
 For all cases below, we assume that none of the earlier cases apply. 
 \begin{enumerate}
 \item If $\preceq_X$ is not a preordering, there is a finite initial binary subsequence $s$ of (the characteristic function of)  $X$ such that no extension of $s$ is a preordering. In this case, let $G_{n,x,y}(X)$ be the length of the least such $s$.
 \item Let $W^X$ be the domain of the well founded part of $\prec_X$ (computable in the data using $\su$), and for each $z \in W^X$ let $f^X_z$ be the characteristic function of $\{w \in W^X : w \prec_X z\}$ and $g^X_z$ be the characteristic function of $\{w \in W^X : w \preceq_X z\}$.
 
 If there is a $\prec_X$-least $w$ such that $g^X_w \neq \max\{f^X_w,F(f^X_w)\}$, we know that $\preceq_X$ differs from $\preceq $, and we need to find a finite approximation to  (the characteristic function of) $X$ where this is manifested. Choose the numerically least such $w$. There will be two sub-cases:
 \smallskip
 
 \noindent - There is a $z$ such that $f^X_w(z) = 0$, $F(f^X_w)(z) = 0$ but $g^X_w(z) = 1$. Select the numerically least such $z$. Then we cannot have both $z \preceq w$ and $w \preceq z$, while we have $z \preceq_X w$ and $w \preceq_X z$, so we let \[G_{n,x,y}(X) = \max\{\langle z,w\rangle,\langle w,z\rangle\} + 1.\]
 \noindent - There is no such $z$. Then there is a $z$ such that $f_w^X(z) = 0$, $F(f^X_w)(z) = 1$, but $g^X_w(z) = 0$. Then we do have $z \preceq w$ and not $z \preceq_X w$, so we let $G_{n,x,y}(X) = \langle z,w \rangle + 1$ for the numerically least such $z$.
 \item If we get to this point, the well-founded part of  $\preceq_X$ is an initial segment of $\preceq$, and we want to decide if this initial segment is proper or not. This is tested by letting $g$ be the characteristic function of $W^X$: the initial segment is proper unless $F(g) \leq g$. If the initial segment is proper, we can find $z$ such that $z \preceq z$ but $z \not \in W^X$, recognised by $g(z) < F(g)(z)$ . If $\langle z,z\rangle \not \in X$, we let $G_{n,x,y}(X) = \langle z,z,\rangle + 1$. If $\langle z,z \rangle \in X$, $z$ is still not in the well founded part of $\prec_X$ so there will be a $w$ such that $w \prec z$ and $w$ is not in the well founded part of $\prec_X$. Since $z$ is of minimal rank in $\prec$ outside $X^W$, we cannot have that $w \prec z$ when $w \not \in X^W$. We can find such $z$ and $w$ using $^2E$ and search over $\N$, and we let $G_{n,x,y}(X) = \max\{\langle z,w\rangle , \langle w,z\rangle\} + 1$.
 \item So far, we have defined $G_{n,x,y}(X)$ independently of $n$, $x$ and $y$. If we have reached this far, we know that $\ci(F) = W^X$, and we let 
 \begin{itemize}
 \item $G_{n,x,y}(X) = n$ if $x \in W^X$, $y \in W^X$ and $x \preceq_X y$
 \item $G_{n,x,y}(X) =0$ otherwise.
 \end{itemize}
 
 \end{enumerate}
 Through items (1) - (3) above, we have constructed $G_{n,x,y}$ such that unless $X$ is a preordering with $\preceq$ as its well founded part, $G_{n,x,y}(X)$ is such that $\preceq$ is not in the neighbourhood of $X$ induced by $G_{n,x,y}(X)$. Moreover $G_{n,x,y}(X)$ is independent of $n$ (and of $x$ and $y$) in this case. We further have that $G_{n,x,y}$ is independent of $n$ if we do not have $x \preceq y$, while the function \[F_{n,x,y}(X) = \left\{\begin{array}{ccc}n & {\rm if} & X = \;\;\preceq \\ 0 & & {\rm otherwise}\end{array}\right.\]  will satisfy the bounding condition induced by $G_{e,x,y}$ if $x \preceq y$. In this case we must have that $M(F_{e,x,y}) \geq n$. This is what we aimed to obtain.
 \end{proof}

\begin{corollary} Let $M$ be a Pincherle realiser that is computable in $\ci$. Then $\ci$ and $({\mathcal H}(M),^2E)$ are computationally equivalent. \end{corollary}
\section{Relativisations}\label{sec8}
It is a matter of routine to relativise concepts of computability to functions $f:\N \rightarrow \N$ or to subsets of $\N$. Our characterisation of the functionals of type 3 computable in  $\ci$ using inductive procedures do relativise directly to objects of type 1. We even had to do so in order to cope with inductive procedures themselves, since there will be function parameters in subcomputations of the form $\{e\}(\ci,\vec a)$. 
\smallskip

In this section we will briefly discuss how the concept of an inductive procedure relativises to parameters $\vec F$ of type 2, without going into any technical details. The key observation is that we can easily extend the definition of procedures to cope with multiple inputs $\vec F$, or, if we are interested in functionals of type 3 computable in a fixed functional $G$ of type 2, to input pairs $F,G$. We only have to add the coordinate of each query when asked during a calculation. 
\smallskip

Given a partial functional $\lambda (F,H)\{e\}(\ci , F,H)$ there will be a procedure $\Omega$ as before, where $\Omega$ and $\Omega_{\rm pre}$ are $\Pi^1_1$, there are functions  {\sc next}, {\sc block}, {\sc denote} and {\sc redenote} computable in $^2E$.  The calculations/strings corresponding to terminating computations will have a blocking accepting the axioms we gave, respecting the rules of bracketing and with no infinitely deep chains of blocks. Our main problem in describing what we mean with an inductive $G$-procedure for a total functional $\lambda F. \{e\}(F,G)$ is to find the right relativisation of $\Pi^1_1$ to $G$. Another, minor problem is that we must allow information about $G$ in the LOG of a calculation, or in some other way, see Remark \ref{remark.G}.
\begin{definition}{\em Let $G$ be a functional of type 2.

\begin{itemize} 
\item[a)] For $g \in \N^\N$, let $\{g^p_i\}_{i \in \N}$ be an enumeration of the set of functions primitive recursive in $g$, where the enumeration is uniformly computable in $g$.
\newline
A \emph{weakly arithmetical formula} $\Phi(g,G)$ is a formula that is arithmetical in $g$ and $\lambda i.G(g^p_i)$.
\item[b)] $X \subseteq \N^\N$ is \emph{weakly arithmetical} in $G$ if it is defined by \[g \in X \leftrightarrow \Phi(g,G)\] for some weakly arithmetical formula $\Phi$.
\item[c)] $X \subseteq \N^\N$ is $\Pi^1_1[G]$ if $X$ is the intersection of a $\Pi^1_1$-set and a set that is weakly arithmetical in $G$.
\end{itemize}

}\end{definition}
\begin{definition}{\em Let $\Omega$ be a procedure for a total functional $\Phi(F)$. Let $G$ be of type 2. $\Omega$ is an \emph{inductive $G$-procedure} if the following are satisfied:
\begin{enumerate}
\item $\Omega$ and $\Omega_{\rm pre}$ are $\Pi^1_1[G]$.
\item There is a function {\sc next}$_\Omega$ partially computable in $^2E$ and terminating on $\Omega_{\rm pre}$.
\item The calculations in  $\Omega$ have blockings, and the blocking structure is guided by the partial functions {\sc bloc}, {\sc denote} and {\sc redenote}, computable in $^2E$ and with the standard properties.
\end{enumerate}}\end{definition}
\begin{theorem} Let $G$ be of type 2 and $\Phi$ of type 3. Then $\Phi$ is computable in $\ci$ and $G$ if and only if $\Phi$ is definable by an inductive $G$-procedure. \end{theorem}
\begin{proof} One way is proved exactly in the same way that we proved that if $\Phi$ is definable by an inductive procedure, then $\Phi$ is computable in $\ci$.
\smallskip

For the other direction, let $\Phi(F) = \{e\}(\ci , F , G)$. We consider the procedure $\Omega^+$ for the partial functional $\lambda F,H.\{e\}(\ci , F , H)$, where both $\Omega^+$ and $\Omega^+_{\rm pre}$ are $\Pi^1_1$ and where the functions {\sc next}, {\sc block} , {\sc denote} and {\sc redenote} are computable in $^2E$. Intersecting with a set that is weakly arithmetical in $G$ we get the procedures in $\Omega^+$ and the strings in $\Omega^+_{\rm pre}$ that are matching $G$. Those will be our $\Omega$ and $\Omega_{\rm pre}$.  \end{proof}
\begin{remark}\label{remark.G}{\em Pairs $(g,G(g))$ will still be present in the calculations. They may be considered to be elements of the LOG of the $G$-calculations. This would actually require a re-definition of our concepts of \emph{procedure} and \emph{calculation}, but we leave how to do it to the reader in case of interest. 

It may be possible to avoid  appearances of pairs   $(g,G(g))$ in the procedure for $\lambda F.\{e\}(F,G)$, but then at the cost of the complexity of the {\sc next}-function and the other functions guiding the blockings. These will then have to be computed by $\ci$ and $G$, and not just by $^2E$.}\end{remark}

\section{Summary and Open problems}\label{sec9}
In this paper we have investigated non-monotone induction as given by a functional $\ci$ of type 3 from the perspective of higher order computability theory. We have established strong closure properties for the \emph{companion} $\LL_\pi$ of the set of functions computable in $\ci$, and we have represented computations relative to $\ci$ and parameters of type 2 in the form of inductive procedures and sequential calculations. Computations relative to $\ci$ can be linearised in a natural way, since application of $\ci$ can be seen as the result of a linear process indexed by some ordinal, and the ordinal rank of a calculation reflects the length of a computation seen as a linear process. We have shown that the length of any terminating computation, with integer inputs in addition to $\ci$, is $\Pi^1_1$-characterisable. There are two open problems related to this:
\begin{problem}\label{problem1}{\em Are all ordinals $\alpha < \pi$ $\Pi^1_1$-characterisable?

}\end{problem}
\begin{problem}\label{problem2}{\em  Are there ordinals computable in $\ci$ that are not the length of any computation $\{e\}(\ci , \vec a)$? }\end{problem}

A positive solution to Problem \ref{problem1} would give us a nice characterisation of the closure ordinal $\pi$, but we conjecture that the answer is negative. We also believe that  when the two problems are solved, the solution will show that they are connected.

Problem \ref{problem2} asks if there is a gap-structure for computing with $\ci$ as it is for infinite time Turing machines, see Hamkins and Lewis \cite{HL} or Welch \cite{Welch}, and for recursion in $^3E$, and not as for computing relative to the Superjump $\mathbb S$. In case there are gaps, it will be of interest to see how the gap structure coincides with the gap structure of infinite time Turing machines computing in time bounded by $\pi$.

This problem also suggests that there is a distinction between various functionals of type 3 similar to the one between predicative and non-predicative arguments in mathematics: in order to compute ${\mathbb S}(F)$ we need to generate the 1-section of $F$, the set of functions computable in $F$, and we need that $F$ is total on its own 1-section, and then we have enough information to deduce what ${\mathbb S}(F)$ will be. This will also work when $F$ is partial, as long as it is not so partial that it is undefined for an input it is able to compute. In order to compute $\ci(F)$ for a partial $F$ we need that $F$ is total on functions computable from $F$ and $\ci$, including the final product of the `computation'. In our main theorem for establishing closure properties of $\pi$, it was essential for the argument that we construct an induction  where $\pi$ tells us to stop, and that we thus must stop before $\pi$. We can only consider partial inductive definitions computable in $\ci$ when they also make sense in the case when the recursion lasts $\pi$ steps in order to deduce that they must stop at an earlier stage. This is a kind of non-predicativity. 
\medskip

We will end this paper with an example of a partial functional $F:\Ca \rightarrow \Ca$ that is computable in $\ci$ and total on the set of $f \in \Ca$ that are computable in $\ci$, but where the closure set $X$ of the associated   inductive definition  is not computable in $\ci$ because $F(X)$ is undefined. 
\begin{example}{\em We define $F$ as a partial function from ${\mathcal P}(\N \times \N)$ to ${\mathcal P}(\N \times \N)$:
\begin{itemize}
\item[-] If $X$ is a well ordering, use Gandy selection for $\ci$ to find an index $e$ for a well ordering $Y \subseteq \N \times \N$ with domain $B$ and  of length extending that of $X$, and let 
\[F(X) = X \cup \{(a,\langle e , b \rangle) : a \in X \wedge b \in B\} \cup\{(\langle e , b\rangle,\langle e , c \rangle) : (b,c) \in Y\}.\]
\item[-]If $X$ is not a well ordering, we let $F(X) = X$.
\end{itemize}
During the induction, a new index $e$ must be found each time, so $F(X_\beta)$ will be an end-extension of $X_\beta$ with a well ordering of the order-type of some $Y_\beta$ with index $e_\beta$ for all $\beta < \pi$. The recursion will stop after $\pi$ steps because then $F(X_\pi)$ is undefined. Clearly, $X_\pi$ is not computable in $\ci$.

}\end{example}

\subsection*{Acknowledgements}

I thank Sam Sanders for involving me in the project this paper is a spin-off of, for reading a preliminary version of this paper, and for giving valuable feedback on the exposition. Our joint project started with him asking me if I could say anything about the computational properties of some weird-looking functionals of type 3. The rest is history.  

I am grateful to John Hartley for his  comments on the exposition.

I am grateful to editors and anonymous referees of other papers from our joint project, their sharp comments often helped me think more clearly about how to present higher order computability in the context of those papers, and then of this one.

I also thank the participants of the seminar on mathematical logic at the University of Oslo for  attending my informal talks on the subjects of this paper, and giving valuable feedback.

\end{document}